\documentclass[11pt]{amsart}
\usepackage{amsmath, amssymb, latexsym}
\usepackage{mathrsfs}
\usepackage{enumerate}

\usepackage{color}

\usepackage[colorlinks=true, pdfstartview=FitV, linkcolor=blue,%
citecolor=blue, urlcolor=blue]{hyperref}

\usepackage[all, knot]{xy}

\newcommand{\nc}{\newcommand}
\newtheorem{Thm}{Theorem}[section]
\newtheorem{Prop}[Thm]{Proposition}
\newtheorem{Cor}[Thm]{Corollary}
\newtheorem{Lem}[Thm]{Lemma}

\theoremstyle{definition}
\newtheorem{Def}[Thm]{Definition}

\theoremstyle{remark}

\newenvironment{red}
{\relax\color{red}}
{\hspace*{.5ex}\relax}

\newcommand{\ber}{\begin{red}}
\newcommand{\er}{\end{red}}

\newenvironment{verd}
{\relax\color{magenta}}
{\hspace*{.5ex}\relax}

\newcommand{\bg}{\begin{verd}}
\newcommand{\eg}{\end{verd}}

\numberwithin{equation}{subsection}

\newcommand{\Z}{\mathbb{Z}}

\newcommand{\C}{\mathbb{C}}

\newcommand{\A}{\mathbb{A}}

\newcommand{\g}{\mathfrak{g}}

\newcommand{\Lb}{\mathfrak{b}}
\newcommand{\h} {\mathfrak{h}}
\newcommand{\Lt} {\mathfrak{t}}
\newcommand{\gl}{\mathfrak{gl}}
\newcommand{\LG}{\mathrm{G}}
\newcommand{\LB}{\mathrm{B}}
\newcommand{\LT} {\mathrm{T}}
\newcommand{\D} {\mathbf{D}}

\newcommand{\gR} {\mathfrak{R}}
\newcommand{\gPol} {\mathfrak{Pol}}
\newcommand{\gS} {\mathbf{S}}
\newcommand{\gP} {\mathbf{P}}
\newcommand{\gK} {\mathbf{K}}
\newcommand{\gH} {\mathbf{H}}

\newcommand{\LH}{\mathrm{H}}
\newcommand{\Lh}{\mathfrak{h}}

\newcommand{\aPol} {\mathsf{Pol}}

\newcommand{\bR} {\mathbf{k}}
\newcommand{\bH} {\mathsf{H}}

\newcommand{\Hom}{\mathsf{ Hom}}

\newcommand{\Ire}{I^{\rm re}}
\newcommand{\Iim}{I^{\rm im}}

\newcommand{\sg}{\mathrm{S}}
\newcommand{\weyl}{\mathrm{W}}
\newcommand{\id}{{e}}
\newcommand{\qin}{{\rm in}}
\newcommand{\qout}{{\rm out}}

\newcommand{\E}{{\mathcal{E}}}

\newcommand{\F}{{\mathcal{F}}}

\newcommand{\tF}{{\widetilde{\mathcal{F}}}}

\newcommand{\vZ}{{\mathcal{Z}}}

\newcommand{\vO}{{\mathcal{O}}}
\newcommand{\ve}{{\mathfrak{e}}}

\newcommand{\im}{{\rm im}}

\newcommand{\eu}{{\rm eu}}

\newcommand{\Ind}{{\rm Ind}}
\newcommand{\Res}{{\rm Res}}

\newcommand{\proj}{\mathsf{pmod}}

\nc{\soplus}{\mathop{\mbox{\normalsize$\bigoplus$}}\limits}
\nc{\sotimes}{\mathop{\mbox{\normalsize$\bigotimes$}}\limits}
\nc{\bl}{\bigl(}
\nc{\br}{\bigr)}
\nc{\tens}{\mathop\otimes\limits}
\nc{\ba}{\begin{array}}
\nc{\ea}{\end{array}}
\nc{\eq}{\begin{eqnarray}}
\nc{\eneq}{\end{eqnarray}}
\nc{\eqn}{\begin{eqnarray*}}
\nc{\eneqn}{\end{eqnarray*}}
\nc{\lan}{\langle}
\nc{\ran}{\rangle}
\nc{\be}{\begin{enumerate}}
\nc{\ee}{\end{enumerate}}
\nc{\bnum}{\be[{\rm(i)}]}
\nc{\noi}{\noindent}
\nc{\bal}[1][w]{\mathfrak{b}_{\alpha,#1}}
\nc{\znu}{\nu^\circ}
\nc{\hs}{\hspace*}
\nc{\vs}{\vspace*}
\nc{\seteq}{\mathbin{:=}}
\nc{\cl}{\colon}
\nc{\To}[1][\quad]{\xrightarrow{\;#1\;}}
\nc{\oO}{\overline{\vO}}
\nc{\bF}{\mathbb{F}}

\newcommand{\isoto}[1][]{\mathop{\xrightarrow[#1]%
{{\raisebox{-.6ex}[0ex][-.6ex]{$\mspace{2mu}\sim\mspace{2mu}$}}}}}

\begin{document}

\title[Geometric realization of KLR algebras]
{Geometric realization of Khovanov-Lauda-Rouquier algebras
associated with Borcherds-Cartan data}

\author[Seok-Jin Kang]{Seok-Jin Kang$^{1}$}
\thanks{$^1$ This work was supported by KRF Grant \# 2007-341-C00001 and NRF Grant \# 2010-0019516.}
\address{Department of Mathematical Sciences and Research Institute of Mathematics,
Seoul National University, 599 Gwanak-ro, Gwanak-gu, Seoul 151-747, Korea}
\email{sjkang@snu.ac.kr}

\author[Masaki Kashiwara]{Masaki Kashiwara$^{2}$}
\thanks{$^2$ This work was supported by Grant-in-Aid for
Scientific Research (B) 22340005, Japan Society for the Promotion of
Science.}
\address{Research Institute for Mathematical Sciences, Kyoto University, Kyoto 606-8502,
Japan and Department of Mathematical Sciences, Seoul National University, 599 Gwanakro,
Gwanak-gu, Seoul 151-747, Korea}
\email{masaki@kurims.kyoto-u.ac.jp}

\author[Euiyong Park]{Euiyong Park$^{3}$}

\thanks{$^3$ This work was supported
by JSPS Postdoctoral Fellowships for Foreign Researchers.}
\address{Department of Pure and Applied Mathematics, Graduate School of Information
Science and Technology, Osaka University, Toyonaka, Osaka 560-0043, Japan}
\email{pwy@cr.math.sci.osaka-u.ac.jp}


\begin{abstract}
We construct a geometric realization of the Khovanov-Lauda-Rouquier
algebra $R$ associated with a  symmetric  Borcherds-Cartan matrix
$A=(a_{ij})_{i,j\in I}$ via quiver varieties.  As an application,
if $a_{ii} \ne 0$ for
any $i\in I$, we prove that there exists a 1-1 correspondence
between Kashiwara's lower global basis (or Lusztig's canonical basis)
of $U_\A^-(\g)$ (resp.\ $V_\A(\lambda)$) and the set of isomorphism
classes of indecomposable projective graded modules over $R$ (resp.\ $R^\lambda$).
\end{abstract}

\maketitle


\vskip 2em

\section*{Introduction}

The {\it Khovanov-Lauda-Rouquier algebras} (or {\it quiver Hecke
algebras}) were introduced independently by Khovanov-Lauda
\cite{KL09, KL11} and Rouquier \cite{R08} to construct a
categorification of quantum groups associated with symmetrizable
Cartan data. For a dominant integral weight $\lambda \in
\mathsf{P}^+$, Khovanov and Lauda conjectured that the cyclotomic
quotient $R^{\lambda}$ of the Khovanov-Lauda-Rouquier algebra $R$
gives a categorification of the irreducible highest weight module
$V(\lambda)$ \cite{KL09}. Recently, this conjecture was proved by
Kang and Kashiwara \cite{KK11}.  In \cite{Web10} Webster also gave a
proof of this conjecture by a completely different method.

When the Cartan datum is symmetric, Varagnolo-Vasserot \cite{VV11}
and Rouquier \cite{R11} gave a geometric realization of
Khovanov-Lauda-Rouquier algebras via quiver varieties
and proved that the isomorphism
classes of projective indecomposable modules correspond to
Kashiwara's lower global basis (or Lusztig's canonical basis)
\cite{Kas91, Lus90}.

In \cite{KOP11}, Kang, Oh and Park introduced a family of
Khovanov-Lauda-Rouquier algebras $R$ associated with symmetrizable
Borcherds-Cartan data and showed that they provide a
categorification of quantum generalized Kac-Moody algebras and their
crystals. More precisely, let $U_\A^{-}(\g)$ be the integral form of
the negative half of the quantum generalized Kac-Moody algebra
$U_q(\g)$ associated with a  Borcherds-Cartan matrix
$A=(a_{ij})_{i,j\in I}$ and let us set $K_{0}(R)= \bigoplus_{\alpha \in
{\mathsf Q}^{+}} K_0(R(\alpha)\text{-}\proj)$, where
$K_0(R(\alpha)\text{-}\proj)$ is the Grothendieck group of the
category $R(\alpha)\text{-}\proj$ of finitely generated projective
graded $R(\alpha)$-modules. Then it was proved in \cite{KOP11} that
there exists an injective bialgebra homomorphism
$$ \Phi : U_\A^- (\g) \hookrightarrow K_{0}(R) $$
and that $\Phi$ is an isomorphism when $a_{ii} \ne 0$ for any $i\in I$. 

A big difference with the case of Kac-Moody algebras is that the
defining relations of $R$ contain a family of polynomials
$\mathcal{P}_i $ of degree $1-\frac{a_{ii}}{2}$ $(i \in I)$ as
twisting factors for commutation and braid relations. As
we will see in Lemma \ref{Lem: computation for Lambdas}, the
polynomials $\mathcal{P}_i$ have a natural geometric interpretation.


For a dominant integral weight $\lambda \in \mathsf{P}^+$, if
$a_{ii} \ne 0 $ for any $i\in I$, it was proved in \cite{KKO11} that
the cyclotomic quotient $R^\lambda$ of $R$ provides a
categorification of the irreducible highest weight $U_q(\g)$-module
$V(\lambda)$. That is, there is a $U_{\A}(\g)$-module isomorphism
$$ \Phi^\lambda : V_\A(\lambda) \buildrel \sim \over \longrightarrow
K_{0}(R^\lambda)\seteq\bigoplus_{\alpha \in \mathsf{Q}^+}
K_0(R^\lambda(\alpha)\text{-}\proj), $$ where
$K_0(R^\lambda(\alpha)\text{-}\proj)$ is the Grothendieck group of
the category of finitely generated projective graded
$R^\lambda(\alpha)$-modules.

In this paper, following the framework of \cite{VV11}, we construct
a geometric realization of Khovanov-Lauda-Rouquier algebras
associated with symmetric Borcherds-Cartan data via quivers
possibly with loops. One of the main ingredients
is Steinberg-type varieties arising from quivers. As
an application, when $a_{ii}\ne 0$ for any $i\in I$, we prove that the
isomorphism $\Phi$ (resp.\ $\Phi^\lambda$) gives a 1-1
correspondence between Kashiwara's lower global basis (or Lusztig's
canonical basis) of $U_\A^{-}(\g)$ (resp.\ $V_\A(\lambda)$) given in
\cite{KS06} and the set of isomorphism classes of indecomposable
projective $R$-modules (resp.\ indecomposable projective
$R^\lambda$-modules).

Let us explain our results more precisely.
Let $Q = (I, \Omega)$ be an
arbitrary locally finite quiver with a vertex set $I$ and an
oriented edge set $\Omega$. The edge set $\Omega$ may have loops
which will give a geometric interpretation of the polynomials
$\mathcal{P}_i$ in the definition of $R$.

Let $\alpha \in \mathsf{Q}^+$ with $|\alpha|=m$ and $I^\alpha = \{
\nu=(\nu_1, \ldots, \nu_m) \in I^m \mid \alpha = \alpha_{\nu_1} +
\cdots + \alpha_{\nu_m} \}$. We fix an $I$-graded vector space
$V_\alpha = \bigoplus_{i\in I} V_i$ with $\underline{\dim}(V_\alpha)
= \alpha$. Let $\E_\alpha$ be the set of all representations of $Q$
with dimension vector $\alpha$, let $\F_\nu$ be the set of complete
flags of type $\nu \in I^\alpha$, and let $ \tF_\nu = \{ (x, F) \in
\E_\alpha \times \F_\nu \mid F \text{ is {\it strictly $x$-stable}}
\}$. Set
$$ \LG_\alpha^\Omega = \LG_\alpha \times \LH^\Omega, $$
where $\LG_\alpha = \prod_{i \in I} GL(V_i) $ and $\LH^{\Omega}$ is
the torus corresponding to the edge set $\Omega$ defined by
$$ \LH^{\Omega} = \prod_{a\in \Omega} \C^*. $$
The group $\LG_\alpha$  acts on $\E_\alpha$ by conjugation  and
transitively on $\F_\nu$,
while $\LH^{\Omega}$ acts on $\E_\alpha$ by multiplication
and trivially on $\F_\nu$.

For $\nu,\nu' \in I^\alpha$, we first define the Steinberg-type
variety
$$\vZ_{\nu,\nu'} = \tF_\nu \times_{\E_\alpha} \tF_{\nu'},$$ and then
consider the convolution algebra $\gR(\alpha)$ and its  polynomial
representation  $\gPol(\alpha)$ as follows:
$$ \gR(\alpha) = \bigoplus_{\nu, \nu' \in I^\alpha}
H_*^{\LG_\alpha^\Omega}(\vZ_{\nu,\nu'}) \langle -2\dim_\C \tF_{\nu} \rangle,\quad
\gPol(\alpha) = \soplus_{\nu \in I^\alpha} H^*_{\LG_\alpha^\Omega}(\tF_{\nu}). $$

In Lemma \ref{Lem: inj of R}, we investigate the properties of the
filtration $ \gR_\alpha^{\le w}$ of $\gR(\alpha)$ given in
$\eqref{Eq: def of Rw}$, and in Lemma \ref{Lem: computation for
Lambdas}, using the commutative diagram $\eqref{Eq: localization
diagram1}$, we compute the equivariant
Euler classes of
fixed points in $\tF_\nu$ and $\vZ^{s_j}_{\nu,\nu'}$,
which will lead us to an explicit description of the
$\gR(\alpha)$-action on $\gPol(\alpha)$.
Here, the polynomials $\mathcal{P}_i$ arise naturally from the
computation of the Euler classes relative to the loops. In
Proposition \ref{Prop: actions of Pol}, we  show that
$\gPol(\alpha)$ is a faithful polynomial representation and give
an explicit description of $\gR(\alpha)$-action on $\gPol(\alpha)$.
It turns out that the cohomology ring
$$\gH = H^*_{\LH^\Omega}( {\rm pt}) \simeq \sotimes_{a\in \Omega}\C[\hbar_a]$$
plays the role of a base ring in the definition of
Khovanov-Lauda-Rouquier algebras.
Using the faithful polynomial representation $\gPol(\alpha)$, we finally prove
that $\gR(\alpha)$ is isomorphic to the Khovanov-Lauda-Rouquier
algebra $R(\alpha)$ (Theorem \ref{Thm: main thm}).

Furthermore, in Proposition \ref{Prop: isom Upsilon}, we investigate
the relation between the category $R(\alpha)$-$\proj$ and the
full subcategory $\mathcal{Q}_\alpha$ of
$\D_{\LG_\alpha^\Omega}^b(\E_\alpha)$. In Theorem \ref{Thm: lgb and
PIM for U} and Corollary \ref{Cor: lgb and PIM for V}, we show that, 
when $a_{ii}\ne 0$ for any $i$, the isomorphism $\Phi$ (resp.\
$\Phi^\lambda$) gives a 1-1 correspondence between Kashiwara's lower
global basis (or Lusztig's canonical basis) of $U_\A^{-}(\g)$
(resp.\ $V_\A(\lambda)$) given in \cite{KS06} and the set of isomorphism
classes of indecomposable projective $R$-modules (resp.\
indecomposable projective $R^\lambda$-modules).

\vskip 2em

\section{Khovanov-Lauda-Rouquier algebras} \label{Sec:KLR}

\subsection{Quantum generalized Kac-Moody algebras}

Let $I$ be an index set. A square matrix $\mathsf{A} =
(a_{ij})_{i,j \in I}$ is called a {\em symmetrizable
Borcherds-Cartan matrix} if it satisfies (i) $a_{ii}=2$ or $a_{ii} \in 2\Z_{\le 0}$ for $i\in I$, (ii) $a_{ij}
\in \Z_{\le 0}$ for $i \ne j$, (iii) $a_{ij} = 0 $ if $a_{ji}=0$ for
$i,j \in I$, (vi) there is a diagonal matrix $\mathsf{D} = {\rm
diag}(\mathsf{d}_i \in \Z_{>0} \mid i\in I)$ such that
$\mathsf{D}\mathsf{A}$ is symmetric. Let $\Ire = \{  i\in I \mid a_{ii} = 2\}$ and $\Iim = I \setminus \Ire$.

A {\it symmetrizable Borcherds-Cartan datum} $(\mathsf{A},
\mathsf{P}, \Pi, \Pi^{\vee})$ consists of
\begin{itemize}
\item[(1)] a symmetrizable Borcherds-Cartan matrix $\mathsf{A}$,
\item[(2)] a free abelian group $\mathsf{P}$, called the {\it weight lattice},
\item[(3)] the set $\Pi = \{ \alpha_i \mid i\in I \} \subset \mathsf{P}$ of {\it simple roots},
\item[(4)] the set $\Pi^{\vee} = \{ h_i \mid i\in I\} \subset \mathsf{P}^{\vee} \seteq\Hom(\mathsf{P}, \Z)$ of {\it simple coroots},
\end{itemize}
which satisfy the following properties:
\begin{itemize}
\item[(i)] $\langle h_i, \alpha_j \rangle \seteq\alpha_j(h_i) = a_{ij}$ for all $i,j\in I$,
\item[(ii)] $\Pi \subset \h^*$ is linearly independent, where $\h\seteq\C \otimes_{\Z} \mathsf{P}^{\vee}$,
\item[(iii)] for each $i \in I$, there exists $\Lambda_i \in \mathsf{P}$ such that $\langle h_j, \Lambda_i \rangle = \delta_{ij}$ for all $j\in I$.
\end{itemize}
We denote by $\mathsf{P}^+ = \{ \lambda \in \mathsf{P} \mid \lambda(h_i) \in
\Z_{\ge 0}, \ i \in I \}$ the set of {\it dominant integral
weights}. The free abelian group $\mathsf{Q} = \bigoplus_{i \in I}
\Z \alpha_i$ is the {\it root lattice},  and $\mathsf{Q}^+ = \sum_{i
\in I} \Z_{\ge 0} \alpha_i$ is the {\it positive root lattice}. For
$\alpha=\sum_{i \in I} k_i \alpha_i \in \mathsf{Q}^+$,  the
\emph{height} $|\alpha|$ of $\alpha$ is $\sum_{i \in I} k_i$. There
is a symmetric bilinear form $(\ | \ )$ on $\h^*$ such that
$$ (\alpha_i | \lambda ) = \mathsf{d}_i \langle h_i, \lambda \rangle
 \ \text{ for } \lambda \in \h^*,\ i\in I.$$
Thus we have $ (\alpha_i| \alpha_j) = \mathsf{d}_i a_{ij} \text{ for } i,j\in I. $

Let $q$ be an indeterminate and $m,n \in \Z_{\ge 0}$. For $i\in \Ire$, let $q_i = q^{\mathsf{d}_i}$ and
\begin{equation*}
 \begin{aligned}
 \ &[n]_i =\frac{ q^n_{i} - q^{-n}_{i} }{ q_{i} - q^{-1}_{i} },
 \ &[n]_i! = \prod^{n}_{k=1} [k]_i ,
 \ &\left[\begin{matrix}m \\ n\\ \end{matrix} \right]_i=  \frac{ [m]_i! }{[m-n]_i! [n]_i! }.
 \end{aligned}
\end{equation*}

\begin{Def} \label{Def: GKM}
The {\em quantum generalized Kac-Moody algebra} $U_q(\g)$ associated
with a Borcherds-Cartan datum $(\mathsf{A},\mathsf{P},\Pi,\Pi^{\vee})$ is the associative
algebra over $\C(q)$ with $1$ generated by $e_i,f_i$ $(i \in I)$ and
$q^{h}$ $(h \in \mathsf{P}^{\vee})$ satisfying following relations:
\begin{enumerate}
  \item  $q^0=1$, $q^{h} q^{h'}=q^{h+h'} $ for $ h,h' \in \mathsf{P}^{\vee}$,
  \item  $q^{h}e_i q^{-h}= q^{\langle h, \alpha_i \rangle} e_i$,
          $q^{h}f_i q^{-h} = q^{-\langle h, \alpha_i \rangle }f_i$ for
$h \in \mathsf{P}^{\vee}, i \in I$,
  \item  $e_if_j - f_je_i =  \delta_{ij} \dfrac{K_i -K^{-1}_i}{q_i- q^{-1}_i }$
, where $K_i=q_i^{ h_i}$,
  \item  $\displaystyle \sum^{1-a_{ij}}_{k=0} (-1)^k\left[\begin{matrix}1-a_{ij} \\ k\\ \end{matrix} \right]_i e^{1-a_{ij}-k}_i
         e_j e^{k}_i =0$ if $i\in \Ire$ and $i \ne j$,
  \item $\displaystyle \sum^{1-a_{ij}}_{k=0} (-1)^k\left[\begin{matrix}1-a_{ij} \\ k\\ \end{matrix} \right]_i f^{1-a_{ij}-k}_if_jf^{k}_i=0$  if $i \in \Ire$ and $i \ne j$,
  \item $ e_ie_j - e_je_i=0$, $f_if_j-f_jf_i =0$ if $a_{ij}=0$.
\end{enumerate}
\end{Def}


We denote by $U_q^-(\g)$ the subalgebra of
$U_q(\g)$ generated by the elements  $f_i$ ($i\in I$).
Let us set $\A =\Z[q,q^{-1}]$ and we denote by $U_{\A}^-(\g)$
the $\A$-subalgebra
of $U_q(\g)$ generated by $f_i^{(n)}$ ($i\in\Ire$ and $n\in\Z_{>0}$)
and $f_i$ ($i\in\Iim$).

For a dominant integral weight $\lambda \in \mathsf{P}^+$,
let $V(\lambda)$ be the irreducible highest weight $U_q(\g)$-module with highest weight $\lambda$
and let $V_\A(\lambda)$ denote
the $U_{\A}^-(\g)$-submodule of $V(\lambda)$ generated
by the highest weight vector.

\vskip 1em

\subsection{Khovanov-Lauda-Rouquier algebras}

For $\alpha \in \mathsf{Q}^+$ with $|\alpha|=m$, let
$ I^\alpha = \{ \nu=(\nu_1, \ldots, \nu_m) \in I^m \mid \alpha_{\nu_1} + \cdots + \alpha_{\nu_m} = \alpha \}$.
The symmetric group $\sg_m = \langle s_k \mid k =1, \ldots, m-1
\rangle$ acts naturally on $I^\alpha$; i.e., for $w\in \sg_m$ and
$\nu =(\nu_1, \ldots, \nu_m) \in I^\alpha$,
\begin{align} \label{Eq: action of Sm}
w \nu = (\nu_{w^{-1}(1)}, \ldots, \nu_{w^{-1}(m)}).
\end{align}

Let $\bR = \soplus_{n \in \Z} \bR_n$ be a commutative graded ring
such that $\bR_n = 0$ for $n < 0$.
 The symmetric group  $\sg_m$ acts on the polynomial ring
$\bR[x_1,\ldots,x_m]$ by
\eqn
&&
w \bl f(x_1,\ldots,x_m)\br = f(x_{w(1)}, \ldots,x_{w(m)})
\quad\text{for $w\in \sg_m$ and $f(x_1,\ldots,x_m) \in \bR[x_1,\ldots,x_m]$.}
\eneqn
For $t=1,\ldots, m-1$, define the
operator $\partial_t$ on $\bR[x_1,\ldots,x_m]$ by
$$ \partial_t(f) = \frac{s_t f - f}{ x_{t} - x_{t+1} } $$
for $f \in \bR[x_1,\ldots,x_m]$.

We take a matrix $\left( \mathcal{Q}_{i,j}(u,v) \right)_{i,j\in I}$
in $\bR[u,v]$ such that $\mathcal{Q}_{i,j}(u,v) = \mathcal{Q}_{j,i}(v,u)$ and $\mathcal{Q}_{i,j}(u,v)$ has the form
\begin{align*}
\mathcal{Q}_{i,j}(u,v) = \left\{
                 \begin{array}{ll}
                   \sum_{p,q \ge 0} t_{i,j;p,q} u^pv^q & \hbox{if } i \ne j,\\
                   0 & \hbox{if } i=j,
                 \end{array}
               \right.
\end{align*}
where $t_{i,j;p,q} \in \bR_{-2(\alpha_i | \alpha_j)- 2\mathsf{d}_ip -2\mathsf{d}_jq}$ and
$t_{i,j;-a_{ij},0}\in \bR_0^\times$.
For each $i\in I$, we choose a polynomial $\mathcal{P}_i(u,v) \in \bR[u,v]$ having the form
\begin{align} \label{Eq: def of Pi}
\mathcal{P}_i(u,v) = \sum_{p,q \ge 0}  h_{i;p,q}  u^pv^q,
\end{align}
where  $h_{i;p,q} \in \bR_{\mathsf{d}_i(2- a_{ii})-2 \mathsf{d}_ip -
2 \mathsf{d}_iq}$ and $h_{i;1- \frac{a_{ii}}{2},0}$, $h_{i;0,1- \frac{a_{ii}}{2}}\in
\bR_0^{\times}$.

\begin{Def}[\cite{KKO11,KOP11}] \label{def:KLR}
Let $(\mathsf{A},\mathsf{P},\Pi,\Pi^\vee)$ be a Borcherds-Cartan datum.
For $\alpha\in \mathsf{Q}^+$ with height $m$,
the {\em Khovanov-Lauda-Rouquier algebra $R(\alpha)$} of weight $\alpha$
 associated with the data $(\mathsf{A},\mathsf{P},\Pi,\Pi^\vee)$,
$(\mathcal{P}_i)_{i\in I}$ and $(\mathcal{Q}_{i,j})_{i,j\in I}$
is the associative graded $\bR$-algebra
generated by $\mathsf{e}(\nu)\ (\nu \in I^\alpha)$,
$\mathsf{x}_k\
(1 \le k \le m)$ and $\mathsf{r}_t\ (1 \le t \le m-1)$ satisfying the
following defining relations:

\begin{align*}
& \mathsf{e}(\nu) \mathsf{e}(\nu') = \delta_{\nu,\nu'} \mathsf{e}(\nu),\ \sum_{\nu \in I^{\alpha}} \mathsf{e}(\nu)=1,\
\mathsf{x}_k \mathsf{e}(\nu) =  \mathsf{e}(\nu) \mathsf{x}_k, \  \mathsf{x}_k \mathsf{x}_l = \mathsf{x}_l \mathsf{x}_k,\\
& \mathsf{r}_t \mathsf{e}(\nu) = \mathsf{e}(s_t( \nu)) \mathsf{r}_t,\  \mathsf{r}_t \mathsf{r}_s = \mathsf{r}_s \mathsf{r}_t \text{ if } |t - s| > 1, \\
&  \mathsf{r}_t^2 \mathsf{e}(\nu) = \left\{
                                                \begin{array}{ll}
                                                  \partial_t\mathcal{P}_{\nu_t}(\mathsf{x}_t,\mathsf{x}_{t+1}) \mathsf{r}_t \mathsf{e}(\nu) & \hbox{ if }  \nu_t = \nu_{t+1}, \\
                                                   \mathcal{Q}_{\nu_t, \nu_{t+1}}(\mathsf{x}_t, \mathsf{x}_{t+1}) \mathsf{e}(\nu) & \hbox{ if } \nu_t \ne \nu_{t+1},
                                                \end{array}
                                              \right. \\
&  (\mathsf{r}_t \mathsf{x}_k - \mathsf{x}_{s_t(k)} \mathsf{r}_t ) \mathsf{e}(\nu) = \left\{
                                                           \begin{array}{ll}
                                                             -  \mathcal{P}_{\nu_t }(\mathsf{x}_t, \mathsf{x}_{t+1}) \mathsf{e}(\nu) & \hbox{if } k=t \text{ and } \nu_t = \nu_{t+1}, \\
                                                               \mathcal{P}_{\nu_t }(\mathsf{x}_t, \mathsf{x}_{t+1}) \mathsf{e}(\nu) & \hbox{if } k = t+1 \text{ and } \nu_t = \nu_{t+1},  \\
                                                             0 & \hbox{otherwise,}
                                                           \end{array}
                                                         \right. \\
&( \mathsf{r}_{t+1} \mathsf{r}_{t} \mathsf{r}_{t+1} - \mathsf{r}_{t} \mathsf{r}_{t+1} \mathsf{r}_{t} )  \mathsf{e}(\nu) \\
& \qquad \qquad = \left\{
                                                                                   \begin{array}{ll}
\mathcal{P}_{\nu_t }(\mathsf{x}_t, \mathsf{x}_{t+2})
\overline{\mathcal{Q}}_{\nu_t,\nu_{t+1}}(\mathsf{x}_t, \mathsf{x}_{t+1},
\mathsf{x}_{t+2})\mathsf{e}(\nu) & \hbox{if } \nu_t = \nu_{t+2} \ne \nu_{t+1}, \\
\overline{\mathcal{P}}_{\nu_t}'( \mathsf{x}_{t}, \mathsf{x}_{t+1}, \mathsf{x}_{t+2})
\mathsf{r}_{t}\mathsf{e}(\nu) +
\overline{\mathcal{P}}_{\nu_t}''( \mathsf{x}_{t}, \mathsf{x}_{t+1}, \mathsf{x}_{t+2}) \mathsf{r}_{t+1}\mathsf{e}(\nu) & \hbox{if } \nu_t = \nu_{t+1} = \nu_{t+2},\\
0 & \hbox{otherwise},
\end{array}
\right.
\end{align*}
where
\begin{align*}
\overline{\mathcal{P}}'_i(u,v,w) &\seteq\frac{\mathcal{P}_{i}(v,u)\mathcal{P}_{i}(u,w)}{(u-v)(u-w)} + \frac{\mathcal{P}_{i}(u,w)\mathcal{P}_{i}(v,w)}{(u-w)(v-w)} - \frac{\mathcal{P}_{i}(u,v)\mathcal{P}_{i}(v,w)}{(u-v)(v-w)}, \\
\overline{\mathcal{P}}''_i(u,v,w) &\seteq -\frac{\mathcal{P}_{i}(u,v)\mathcal{P}_{i}(u,w)}{(u-v)(u-w)} - \frac{\mathcal{P}_{i}(u,w)\mathcal{P}_{i}(w,v)}{(u-w)(v-w)} + \frac{\mathcal{P}_{i}(u,v)\mathcal{P}_{i}(v,w)}{(u-v)(v-w)}, \\
\overline{\mathcal{Q}}_{i,j}(u,v,w)
&\seteq \frac{\mathcal{Q}_{i,j}(u,v) -
\mathcal{Q}_{i,j}(w,v)}{u-w}.
\end{align*}
\end{Def}


For $\nu \in I^{\alpha}$, set
$ \mathsf{x}_k(\nu) = \mathsf{x}_k \mathsf{e}(\nu)$ and $ \mathsf{r}_t(\nu) = \mathsf{r}_t \mathsf{e}(\nu) $.
Then the $\Z$-grading on
$R(\alpha)$ is given by
\begin{align*}
\deg(\mathsf{e}(\nu))=0, \quad \deg(\mathsf{x}_k (\nu))= 2  \mathsf{d}_{\nu_k}, \quad  \deg(\mathsf{r}_t(\nu))= -(\alpha_{\nu_{t}} | \alpha_{\nu_{t+1}}).
\end{align*}

For $\lambda \in \mathsf{P}^+$ and $i\in I$,
let us choose a polynomial $a_i^\lambda(u)$ of the from
$$a_i^\lambda(u)=\sum_{k=0}^{\lambda(h_i)}
c_{i;k}^\lambda u^{\lambda(h_i)-k}$$
where $c_{i;k}^\lambda\in\bR_{2\mathsf{d}_ik}$ and $c_{i;0}^\lambda=1$.
We set $ a^\lambda(\mathsf{x})=\sum_{\nu\in I^\alpha}
a_{\nu_1}^\lambda(\mathsf{x}_1)\mathsf{e}(\nu) $.
Then the {\it cyclotomic Khovanov-Lauda-Rouquier algebra} $R^{\lambda}(\alpha)$ at $\alpha$ is defined to be the quotient algebra
$$ R^{\lambda}(\alpha) = R(\alpha) / R(\alpha) a^\lambda(\mathsf{x})R(\alpha). $$

Assume that $\bR_0$ is a field.
Let
\begin{align}  \label{Eq: Grothedieck gp}
K_0(R)=\bigoplus_{\alpha \in \mathsf{Q}^+} K_0(R(\alpha)\text{-}\proj), \quad
K_0(R^{\lambda})=\bigoplus_{\alpha \in \mathsf{Q}^+} K_0(R^{\lambda}(\alpha)\text{-}\proj),
\end{align}
where $K_0(R(\alpha)\text{-}\proj)$ (resp.\
$K_0(R^{\lambda}(\alpha)\text{-}\proj)$) is the Grothendieck group
of the category $R(\alpha)\text{-}\proj$ (resp.\
$R^\lambda(\alpha)\text{-}\proj$) of finitely generated projective
graded left $R(\alpha)$-modules (resp.\ finitely generated
projective graded left $R^\lambda(\alpha)$-modules). Then $K_0(R)$
(resp.\ $K_0(R^\lambda)$) has the $\A$-module structure induced by
the $\Z$-grading on $R$ (resp.\ $R^\lambda$), where $\A\seteq\Z[q,q^{-1}]$.

For $\beta,\beta' \in \mathsf{Q}^{+}$, we define
\begin{align*}
\mathsf{e}(\beta,\beta') = \sum_{\nu\in I^{\beta}, \nu'\in
I^{\beta'}}  \mathsf{e}(\nu, \nu')\in R(\beta+\beta'),
\end{align*}
where  $\mathsf{e}(\nu, \nu') $ is the idempotent corresponding to
the concatenation of $\nu$ and $\nu'$. For each $i\in I$ and
$\alpha,\beta \in \mathsf{Q}^+$, we define the functors
\eq&&\ba{rl}
\Res_{\alpha,\beta} &: R(\alpha+\beta)\text{-}\proj \longrightarrow R(\alpha)\otimes R(\beta)\text{-}\proj, \\
\Ind_{\alpha,\beta} &: R(\alpha) \otimes R(\beta) \text{-}\proj \longrightarrow R(\alpha+\beta)\text{-}\proj
\ea\eneq
by
\begin{align*}
&\Res_{\alpha,\beta}(N)  = \mathsf{e}(\alpha,\beta)N, \\
&\Ind_{\alpha,\beta}(L) = R(\alpha+\beta) \mathsf{e}(\alpha,\beta) \tens_{R(\alpha)\otimes R(\beta)} L
\end{align*}
for $L \in R(\alpha)\otimes R(\beta)\text{-}\proj $
and $N \in R(\alpha+\beta)\text{-}\proj$.  One can show
that $K_0(R(\alpha)\text{-}\proj )\otimes_\A K_0(R(\beta)\text{-}\proj )
\to K_0( R(\alpha)\otimes R(\beta)\text{-}\proj)$ is an isomorphism, and
$K_{0}(R)$ becomes a bialgebra
\cite{KOP11,KL09}.

For a $\Z$-graded module $M=\bigoplus_{k\in \Z}M_k$ and $t\in
\Z$, let $M\lan t\ran$ be the $\Z$-graded module
defined by $M\langle t \rangle_k = M_{k+t}$.
For each $i\in\Ire$ and $m\ge 0$,
we define the projective graded $R(m \alpha_i)$-module
$P(i^{m})$ to be
$$P(i^{m})=\dfrac{R(m\alpha_i)}
{\sum_{t=1}^{m-1}R(m\alpha_i)\mathsf{r}_t}
\big\langle \frac{m(m-1)(\alpha_i|\alpha_i)}{4}\big\rangle.$$
If $i \in
\Iim$, we define $P(i)=R(\alpha_i)$, the regular representation. The
following theorem was proved in \cite{KOP11} (see also \cite{KL09}).

\begin{Thm}[\cite{KOP11}] \label{Thm: categorification U}
\bnum
\item There exists an injective bialgebra homomorphism $\Phi: U_\A^-(\g) \longrightarrow K_0(R)$
sending $f_i^{(m)}$ to $P(i^{m})$ $(i\in \Ire)$ and $f_i$ to
$P(i)$ $(i \in \Iim)$.

\item If $a_{ii}\ne 0$ for any $i\in I$, then $\Phi$ is an isomorphism.
\end{enumerate}
\end{Thm}

For a dominant integral weight $\lambda \in \mathsf{P}^+$, we define
the functors
\begin{align*}
E_i^\lambda &: R^\lambda(\beta + \alpha_i)\text{-}\proj \longrightarrow R^\lambda(\beta)\text{-}\proj, \\
F_i^\lambda &: R^\lambda(\beta )\text{-}\proj \longrightarrow R^\lambda(\beta+ \alpha_i)\text{-}\proj
\end{align*}
by
\begin{align*}
E_i^\lambda(N)  = \mathsf{e}(\beta,\alpha_i)N, \qquad
F_i^\lambda(M) = R^\lambda(\beta+ \alpha_i) \mathsf{e}(\beta,\alpha_i) \otimes_{R^\lambda(\beta)}M
\end{align*}
for $M \in R^\lambda(\beta)\text{-}\proj $ and $N \in
R^\lambda(\beta + \alpha_i)\text{-}\proj$, respectively. Then the
functors $E_i^{\lambda}$, $F_i^{\lambda}$ define a
$U_{\A}(\g)$-module structure on $K_{0}(R^\lambda)$. Let
$V_\A(\lambda)$ be the $\A$-form of the irreducible highest weight
$U_q(\g)$-module $V(\lambda)$. The following generalized version of
{\it cyclotomic categorification conjecture} was proved in
\cite{KKO11}.

\begin{Thm}[\cite{KK11,KKO11}] \label{Thm: categorification V}
If $a_{ii}\ne 0$ for any $i\in I$, then there exists a $U_\A(\g)$-module isomorphism $ \Phi^\lambda : V_\A(\lambda) \buildrel \sim \over \longrightarrow K_0(R^{\lambda}) $.
\end{Thm}

In this paper,
in order to give a geometric realization of $R(\alpha)$,
we assume that {\em $\mathsf{A}$ is symmetric.}

Then, we can take $\mathrm{d}_i = 1$ and
$(\alpha_i | \alpha_j) = a_{ij}$ for $i \ne j\in I$. Set $\mathsf{l}_i =
1-a_{ii}/2$ for $i \in I$ and choose non-negative integers
$\mathsf{h}_{ij}$ such that $a_{ij}= -\mathsf{h}_{ij} -
\mathsf{h}_{ji}$ for $i \neq j\in I$. Let
$\Lambda_{i,j}$ ($i,j\in I$) be an
index set of
$\mathsf{h}_{i,j}$ elements for $i\not=j$
and of $\mathsf{l}_{i}$ elements for $i=j$.
We set $\Lambda = \bigsqcup_{i,j\in I} \Lambda_{i,j}$. Let
$\bH$ be the polynomial ring over $\C$ generated by indeterminates
$\hbar_{a}$ ($a\in \Lambda$) with $\deg( \hbar_a) = 2$;
i.e.,
\begin{align} \label{Eq: def of bH}
\bH \simeq  \bigotimes_{a \in \Lambda} \C[\hbar_a].
\end{align}
We take $\bH $ as the base ring $\bR$ and consider the specially chosen
polynomials $ \mathcal{P}_i(u,v)$ and $\mathcal{Q}_{i,j}(u,v)$ in
$\bH[u,v]$ given as follows:

\begin{equation} \label{Eq: P and Q}
\begin{aligned}
\mathcal{P}_i(u,v) &=  \prod_{a \in \Lambda_{i,i}}(u-v+\hbar_a), \\
\mathcal{Q}_{i,j}(u,v) &= \left\{
   \begin{array}{ll}
     0 & \hbox{ if } i =j, \\
     \prod_{a \in \Lambda_{i,j}} (v-u + \hbar_a)
     \prod_{a \in \Lambda_{j,i}} (u-v + \hbar_a) & \hbox{ if }  i \ne j.
   \end{array}
 \right.
\end{aligned}
\end{equation}
 In this case, the algebras $R(\alpha)$ has another $\Z$-grading
given as follows:
\begin{align} \label{Eq: grading of R}
\deg(\mathsf{e}(\nu))=0, \quad \deg(\mathsf{x}_k (\nu))= 2, \quad  \deg(\mathsf{r}_l(\nu))= \left\{
                                                                                              \begin{array}{ll}
                                                                                                2(\mathsf{l}_{\nu_l}-1) & \hbox{ if } \nu_l = \nu_{l+1}, \\
                                                                                                2\mathsf{h}_{\nu_{l}, \nu_{l+1}} & \hbox{ if } \nu_l \ne \nu_{l+1}.
                                                                                              \end{array}
                                                                                            \right.
\end{align}

We now construct a faithful graded polynomial representation of $R(\alpha)$ associated with $ \mathsf{A}$, $\eqref{Eq: P and Q}$ and $\eqref{Eq: grading of R}$.
Set
\begin{align*}
\aPol(\alpha) = \bigoplus_{\nu \in I^{\alpha}} \bH[\mathsf{x}_1(\nu), \ldots, \mathsf{x}_m(\nu)].
\end{align*}
with $\deg(\mathsf{x}_k(\nu))=2$.
For $f \in \bH[x_1,\ldots,x_m]$, we denote by
$f(\nu)$ the element $f(\mathsf{x}_1(\nu),\ldots,\mathsf{x}_m(\nu)) \in
\bH[\mathsf{x}_1(\nu), \ldots, \mathsf{x}_m(\nu)]$.

Then the following proposition is proved in \cite{KOP11,KL09,R08}.
\begin{Prop} \label{Prop: faithful}
The algebra $\aPol(\alpha)$ has a graded $R(\alpha)$-module structure
as follows:
for $\nu=(\nu_1,\ldots,\nu_m),\ \nu' \in I^{\alpha}$ and
$f\in \bH[{x}_1, \ldots,{x}_m]$,

\begin{equation} \label{Eq: actions in aPol}
\begin{aligned}
\mathsf{e}(\nu) \cdot f(\nu') &= \delta_{\nu,\nu'} f(\nu), \\
\mathsf{x}_k(\nu) \cdot f(\nu') &= \delta_{\nu,\nu'} \mathsf{x}_k(\nu)f(\nu),\\
\mathsf{r}_l(\nu) \cdot f(\nu') &= \left\{
                                     \begin{array}{ll}
                       0 & \hbox{ if } \nu \ne \nu', \\
       \Bigl( \prod_{a\in \Lambda_{\nu_l,\nu_{l+1} }} (\mathsf{x}_{l}(s_l\nu)-\mathsf{x}_{l+1}(s_l\nu)+\hbar_a )\Bigr) (s_lf)(s_l\nu) & \text{if $\nu = \nu'$,
$s_l\nu \ne \nu$,} \\
 \Bigl( \prod_{a\in \Lambda_{\nu_l,\nu_{l} }} (\mathsf{x}_{l}( \nu)-\mathsf{x}_{l+1}( \nu)+ \hbar_a)\Bigr) (\partial_l f)(\nu) & \text{if $\nu = \nu'$, $s_l\nu = \nu$.}
                                     \end{array}
                                   \right.
\end{aligned}
\end{equation}
Moreover, $\aPol(\alpha)$ is a faithful $R(\alpha)$-module.
\end{Prop}

\vskip 2em

\section{Geometric realization} \label{Sec:realization}

In this section, we give a geometric realization of $R(\alpha)$. We
first construct the Steinberg-type variety $\vZ_{\nu,\nu'}$ out of a
given locally finite quiver.  We then investigate the convolution
algebra  $\gR(\alpha) = \bigoplus_{\nu, \nu' \in I^\alpha}
H_*^{\LG_\alpha^\Omega}(\vZ_{\nu,\nu'}) \langle -2\dim_\C \tF_{\nu}
\rangle $ arising from $\vZ_{\nu,\nu'}$ and show that the algebra
$\gR(\alpha)$ is isomorphic to the Khovanov-Lauda-Rouquier algebra
$R(\alpha)$ using the faithful representations $\aPol(\alpha)$ and
$\gPol(\alpha)$ given in Proposition \ref{Prop: faithful} and
Proposition \ref{Prop: actions of Pol}.

\subsection{Quiver representation varieties}

Let $Q=(I, \Omega)$ be an arbitrary locally finite quiver with a
vertex set $I$ and an oriented edge set $\Omega$. For  $a \in
\Omega$, let $\qin(a)$ (resp.\ $\qout(a)$) denote the incoming
vertex (resp.\ the outgoing vertex) of $a$. For $i,j, k\in I$ with
$j \neq k$, let
\begin{align*}
& \Omega_{i,j} = \{ a\in \Omega \mid \qout(a)=i,\, \qin(a)=j  \},  \\
&  \ell_{i} = \# \Omega_{i,i} \ \text{ and }\  h_{j,k} = \#\Omega_{j,k}.
\end{align*}

Let $U_q(\g)$ be the quantum generalized Kac-Moody algebra associated with the Borcherds-Cartan matrix $\mathsf{A}=(a_{ij})_{i,j\in I }$ defined by
$$ a_{ii} = 2 - 2 \ell_{i}, \quad a_{jk} =  -h_{jk} -h_{kj}  $$
for $i,j,k \in I$ with $j \neq k$. Note that $\Ire = \{ i \in I \mid
\ell_i = 0 \}$,  $\Iim = \{ i\in I \mid \ell_i \ne 0 \}$.

%

Define the torus $\LH^{\Omega}$ corresponding to the oriented edge set $\Omega$ by
$$ \LH^{\Omega} = \prod_{a\in \Omega} \C^*
$$
and denote its Lie algebra by $\Lh^\Omega$.

Let $\alpha \in \mathsf{Q}^+$ and let $m = |\alpha|$. Fix an
$I$-graded $\C$-vector space $V_\alpha = \bigoplus_{i\in I} V_i $
with $\underline{\dim} V_\alpha  \seteq \sum_{i\in I} (\dim V_i)
\alpha_i = \alpha$. Let $\LG_i = GL(V_i)$ ($i\in I$) and $\LG_\alpha
= \prod_{i\in I} GL(V_i)$. We denote by $\g_\alpha = \bigoplus_{i\in
I}\gl(V_i)$ the Lie algebra of $\LG_\alpha$ and identify $\g_\alpha$
with its dual $\g_\alpha^*$ via $\sum_{i\in I} \mathrm{tr}_{V_i}$.
Let
$$ \LG_\alpha^\Omega = \LG_\alpha \times \LH^\Omega. $$
We define the {\it quiver representation variety} to be
\begin{align*}
\E_\alpha = \soplus_{a\in \Omega} \Hom(V_{\qout(a)}, V_{\qin(a)}).
\end{align*}
The group $\LG_\alpha^\Omega$ acts on $\E_\alpha$ by
\begin{align} \label{Eq: G action}
(g,t) \cdot x = \bl t_a  ( g_{\qin(a)}  x_a  g_{\qout(a)}^{-1}) \br_{a\in \Omega}
\end{align}
for $x = (x_a)_{a\in \Omega} \in \E_\alpha$ and $(g,t) = ((g_i)_{i\in I},  (t_a)_{a\in \Omega})\in \LG_\alpha^\Omega$.

\smallskip
For $\nu=(\nu_1,\ldots,\nu_m) \in I^\alpha$,
let
\begin{align*}
 \F_\nu = \bigl\{ F = (0 = & F_0 \subset F_1  \subset \cdots \subset F_m = V_\alpha)  \\
& \mid \text{$F_k$ is $I$-graded and
$\underline{\dim}(F_k /F_{k-1}) = \alpha_{\nu_k}$ for $k=1,\ldots,m$} \bigr\}
\end{align*}
be the variety of {\it complete flags of type $\nu$} and set $\F_\alpha
= \bigsqcup_{\nu \in I^{\alpha}} \F_\nu$.
For a vector space $W$, we denote by $\F(W)$
the variety consisting of complete flags of $W$.
Then there is an isomorphism
\begin{align} \label{Eq: iso of F}
\F_\nu \buildrel \sim \over \longrightarrow \prod_{i\in I} \F(V_i),
\qquad  F \mapsto ( ^iE   )_{i\in I},
\end{align}
where $^i E_k = V_i \cap F_k $ for $k=1,\ldots,m$ and $ ^i E = (0 = {^iE_0} \subset {^iE_1} \subset \cdots \subset {^iE_m} = V_i )\in\F(V_i)$ for $i\in I$.
The group $\LG_\alpha$ acts transitively
on $\F_\nu$ and $\prod_{i\in I} \F(V_i)$, and $\LH^\Omega$ acts trivially on them,
so that the isomorphism $\eqref{Eq: iso of F}$ is
$\LG_\alpha^\Omega$-equivariant.

For $x \in \E_\alpha$ and $F =(0 =F_0 \subset F_1\subset \cdots \subset F_m) \in \F_\nu$ is said to be {\em strictly $x$-stable} if
$x_a(F_k \cap V_{\qout(a)}) \subset F_{k-1}  \cap V_{\qin(a)} $ for all $a\in \Omega$ and $ k=1,\ldots,m$.
Let
\begin{align*}
 \tF_\nu = \{ (x, F) \in \E_\alpha \times \F_\nu \mid
\text{ $F$ is strictly $x$-stable}    \}
\end{align*}
and let $\tF_\alpha = \bigsqcup_{\nu \in I^{\alpha}} \tF_\nu$. Then
$\LG_\alpha^\Omega$ acts diagonally on $\tF_\nu$ and $\tF_\nu$ is a
$\LG_\alpha^\Omega$-equivariant  vector bundle over $\F_\nu$.

Let us consider the following maps
\eq \label{Eq: maps pi and p}
&&\ba{c}\xymatrix{
   & \ar[ld]_{\pi_\nu}  \tF_{\nu}  \ar[rd]^{\mathrm{p}_\nu}  & \\
 \E_\alpha & & \F_\nu
}\ea
\eneq
defined by $\pi_\nu(x, F) = x $ and $\mathrm{p}_\nu(x,F)=F$ for
$(x,F) \in \tF_{\nu}$. Both of the maps $\pi_\nu$ and
$\mathrm{p}_\nu$ are $\LG_\alpha^\Omega$-equivariant  and $\pi_\nu$
is projective. For $\nu, \nu' \in I^\alpha$, we define the
Steinberg-type variety
\begin{align*}
\vZ_{\nu, \nu'} & = \tF_{\nu} \times_{\E_\alpha} \tF_{\nu'} \\
& \simeq \{ (x, F, F') \in \E_\alpha \times \F_{\nu} \times \F_{\nu'} \mid  (x, F) \in \tF_{\nu}, (x, F') \in \tF_{\nu'} \},
\end{align*}
and let $ \vZ_{\alpha} = \bigsqcup_{\nu, \nu' \in I^\alpha} \vZ_{\nu, \nu'}$.
Consider the map
\begin{align} \label{Eq: map q1}
\mathrm{q}_{\nu, \nu'} : \vZ_{\nu, \nu'} \longrightarrow \F_{\nu} \times \F_{\nu'}
\end{align}
defined by $\mathrm{q}_{\nu, \nu'}(x, F, F')= (F, F')$ for $(x, F, F') \in \vZ_{\nu, \nu'}$, and set
\begin{align} \label{Eq: map q2}
\mathrm{q}_{\alpha} \seteq \bigsqcup_{\nu,\nu' \in I^\alpha} \mathrm{q}_{\nu, \nu'}: \vZ_{\alpha} \longrightarrow \F_{\alpha} \times \F_{\alpha}.
\end{align}
Then  the map $\mathrm{q}_{\nu, \nu'}$ is a
$\LG_\alpha^\Omega$-equivariant morphism.


\bigskip
From now on, we fix $\znu= (\znu_1, \ldots, \znu_m) \in
I^\alpha$ and $ F^\circ\in \F_{ \znu}$. Set $\LG = GL(V_\alpha)$
and let $\LB$ be the stabilizer of $F^\circ$ in $\LG$. Then $\LB$
and $\LB_\alpha \seteq \LB \cap \LG_\alpha$ are Borel subgroups of $\LG$
and $\LG_\alpha$, respectively. Take a maximal torus $\LT$ of
$\LB_\alpha$ and let
$$ \LT^\Omega = \LT \times \LH^\Omega. $$
We denote by $\g$ (resp.\ $\Lb$, $\Lb_\alpha$, $\Lt$) the Lie
algebra of $\LG$ (resp.\ $\LB$, $\LB_\alpha$, $\LT$).

There exist
1-dimensional $\LT$-submodules $\mathbb{F}_1$, \ldots,
$\mathbb{F}_m$ of $V_\alpha$ such that
\begin{align} \label{Eq: def of 1 dim F}
F^\circ = ( 0 \subset  \mathbb{F}_1 \subset \mathbb{F}_1
\oplus \mathbb{F}_2 \subset \cdots \subset \mathbb{F}_1\oplus \cdots
\oplus\mathbb{F}_m = V_{\alpha} ).
\end{align}
Note that $\bF_k\subset V_{\nu^\circ_k}$ for $k=1,\ldots,m$.
For $w, w' \in \sg_m$, set $\nu_w =
(\znu_{w(1)}, \ldots, \znu_{w(m)})$ and
\begin{align} \label{Eq: def of Fw}
&F_w = ( 0  \subset \mathbb{F}_{w(1)} \subset \mathbb{F}_{w(1)} \oplus \mathbb{F}_{w(2)} \subset \cdots \subset \mathbb{F}_{w(1)}\oplus \cdots \oplus\mathbb{F}_{w(m)} = V_\alpha) \in \F_{\nu_w},\\
&F_{w,w'} = (F_w, F_{w'}) \in \F_{\nu_w} \times \F_{\nu_{w'}}.
\end{align}
Let $\LB_{\alpha,w}$ be the stabilizer of $F_w$ in $\LG_\alpha$.
 Let $\weyl$ and $\weyl_\alpha$ be the Weyl group of the pair $(\LG, \LT)$
and $(\LG_\alpha, \LT)$, respectively.
We identify $\weyl$ with the symmetric group $\sg_m$ via $w \cdot \mathbb{F}_k = \mathbb{F}_{w(k)}$
for $w\in \weyl$ and $k=1,\ldots,m$; i.e., $w \cdot F^\circ = F_w$. Note that
 $w \nu_{w'} = \nu_{w'w^{-1}} $ for $w,w' \in \weyl$ by $\eqref{Eq: action of Sm}$.
The quotient set $ \weyl_\alpha \backslash \weyl  $ is in
1-1 correspondence with $ I^\alpha $ via the map sending $w$ to $
\nu_{w}$.

For $w\in \weyl$, set
\begin{align} \label{Eq: def of ew}
\mathfrak{e}_w = \{ x \in \E_\alpha  \mid F_{w}  \text{ is strictly $x$-stable} \}
\simeq \mathrm{p}_{\nu_w}^{-1}(F_w).
\end{align}
By the definition, we have
\eq
&&\ba{rcl}
 \ve_w&=& \{ x \in \E_\alpha  \mid
x ( \mathbb{F}_{w(k)}) \subset \mathbb{F}_{w(1)}
\oplus \cdots \oplus \mathbb{F}_{w(k-1)} \text{ for $1 \le k \le m$}    \}\\
&\simeq&\soplus_{m\ge k>k'\ge1}\mathrm{Hom}( \mathbb{F}_{w(k)} , \mathbb{F}_{w(k')} )
^{\oplus\Omega_{(\nu_w)_k,(\nu_w)_{k'}}} .
\ea\label{eq:cw}
\eneq
Similarly, the Lie algebra $\bal$ of $B_{\alpha,w}$ is calculated as
\eq
&&\ba{rcl}
 \bal&=& \{ x \in \soplus_{i\in I}\mathrm{End}(V_i)\mid
x ( \mathbb{F}_{w(k)}) \subset \mathbb{F}_{w(1)}
\oplus \cdots \oplus \mathbb{F}_{w(k)} \text{ for } 1 \le k \le m  )  \}\\
&\simeq&
\soplus_{\substack{m\ge k\ge k'\ge1,\\(\nu_w)_k=(\nu_w)_{k'}}}
\mathrm{Hom}( \mathbb{F}_{w(k)} , \mathbb{F}_{w(k')} ).
\ea\label{eq:bw}
\eneq

For $w\in \weyl$ and $\nu,\nu' \in I^\alpha$, we define
\begin{alignat*}{2}
\vO^w_{\alpha} &= \LG \cdot F_{e,w} \cap ( \F_{\alpha} \times \F_{\alpha}), &
\quad \overline{ \vO}^w_{\alpha} &= \text{%
the closure of $\vO^w_{\alpha} $ in $\F_{\alpha} \times \F_{\alpha}$,}\\
\vO^w_{\nu, \nu'}& = \vO^w_{\alpha} \cap ( \F_{\nu} \times \F_{\nu'}),& \quad
\quad \overline{ \vO}^w_{\nu, \nu'} &= \overline{\vO}^w_{\alpha} \cap ( \F_{\nu} \times \F_{\nu'}).
\end{alignat*}
 Note that we have a Bruhat decomposition
$ \F_\alpha \times \F_\alpha =
\bigsqcup_{w\in \weyl}\vO^w_{\alpha}$.
For $w,w',w''\in \weyl$, we abbreviate
$$ \vO^w_{w', w''} = \vO^w_{\nu_{w'}, \nu_{w''}}, \quad \quad
\overline{ \vO}^w_{w', w''} = \overline{ \vO}^w_{\nu_{w'},
\nu_{w''}}.$$

\begin{Lem} \label{Lem: properties of O}
Let $w \in \weyl$, $\nu, \nu' \in I^\alpha$ and let $s_i \in \weyl $ be a simple reflection.
\bnum
\item The set of elements of $\vO^w_\alpha$ fixed under the diagonal action of $\LT$ is $\{ F_{u,uw} \mid u\in \weyl \}$.\label{Oitem1}

\item  If $\nu'\not= w^{-1}\nu$, then $ \overline{\vO}^w_{\nu,\nu'} = \emptyset$.
\label{Oitem2}

\item
We have
$$\overline{\vO}^{s_i}_{\alpha}
=\bigsqcup_{\nu\in I^\alpha}
\{(F,F')\in \F_\nu\times\F_{s_i\nu}
\mid F_k=F'_k\quad\text{for any $k\not=i$}\}.$$
In particular, $\overline{\vO}^{s_i}_{\nu, s_i\nu}$
 is smooth and $\dim_{\C} \overline{\vO}^{s_i}_{\nu, s_i\nu}
=\dim_{\C} \F_\nu+\delta_{\nu,s_i\nu}$.
\label{Oitem3}
\item  If $\nu= s_i\nu$, then
$\overline{\vO}^{s_i}_{\nu,\nu} = {\vO}^{s_i}_{\nu,\nu}\cup {\vO}^{e}_{\nu,\nu}$,
and the first and the second projections
$\overline{\vO}^{s_i}_{\nu,\nu}\to\F_\nu$ are $\mathbb{P}^1$-bundles.
\label{Oitem4}
\item If $\nu\ne s_i \nu$, then
$ \overline{\vO}^{s_i}_{\nu,s_i\nu} = {\vO}^{s_i}_{\nu,s_i\nu}$ and
$\overline{\vO}^{s_i}_{\nu,\nu} = \emptyset$.
Moreover the projections
$\overline{\vO}^{s_i}_{\nu,s_i\nu}\to \F_\nu$
and $\overline{\vO}^{s_i}_{\nu,s_i\nu}\to\F_{s_i\nu}$ are isomorphisms.
\label{Oitem5}
\item If $w^{-1}\nu = \nu'$, then $\vO^w_{\nu,\nu'}$
is a $\LG_\alpha$-orbit and it is an affine fibration over $\F_{\nu}$
with fiber being isomorphic to the affine space $\LB_{\alpha,u}\cdot F_{uw} $
for $u\in W$ such that $\nu=\nu_{u}$.
\label{Oitem6}
\end{enumerate}
\end{Lem}
\begin{proof}
We shall only prove \eqref{Oitem2}
and \eqref{Oitem6}, since the other assertions
are easily proved (see \cite[Lemma 2.6]{VV11}).
Since $\vO^w_{\nu,\nu'}$ is $\LG_\alpha$-invariant, $\vO^w_{\nu,\nu'}$ is a union of $G_\alpha$-orbits, and any $\LG_\alpha$-orbit of
$\F_\alpha \times \F_\alpha$ contains a $\LT$-fixed point.

We first show the assertion \eqref{Oitem2}. Suppose that
$w^{-1}\nu \ne \nu'$. If $\overline{\vO}^w_{\nu,\nu'}$ is not empty, then $\vO^w_{\nu,\nu'}$ contains a $\LT$-fixed point $F_{u,uw}$ for some $u\in W$.
Then we have $\nu = \nu_{u}$ and $\nu' = \nu_{uw}=w^{-1}\nu_{u}$,
which is a contradiction.

Now consider the case $w^{-1}\nu = \nu'$. Choose a $\LT$-fixed point $F_{u,uw} \in \vO^w_{\nu,\nu'}$ for some $u\in W$. Then $\nu = \nu_u$ and $\nu' = \nu_{uw}$.
If we take another $\LT$-fixed point $F_{u',u'w} \in \vO^w_{\nu,\nu'}$, then there is an element $v\in \weyl_\alpha$
such that
$$v \cdot F_{u} = F_{vu} = F_{u'},$$
which yields $F_{u',u'w} = v \cdot F_{u,uw} \in \LG_\alpha \cdot F_{u,uw}$.
Therefore,
$$ \vO^w_{\nu,\nu'} = \LG_\alpha \cdot F_{u,uw}. $$

Let $\phi^w_{\nu,\nu'}: \vO^w_{\nu,\nu'} \rightarrow \F_{\nu} $ be the map defined by
$\phi^w_{\nu,\nu'}(F,F') = F$ for $(F,F') \in \vO^w_{\nu,\nu'}$. Then $\phi^w_{\nu,\nu'}$ is surjective since $\phi^w_{\nu,\nu'}$ is $\LG_\alpha$-equivariant, and
the fiber of $g \cdot F_{u}$ is given as follows:
\begin{align*}
 {\phi^w_{\nu,\nu'}}^{-1}(g \cdot F_{u}) &= g \cdot \{ (F_{u}, F') \mid F'
\in \LB_{\alpha, u} \cdot F_{uw} \} \\
 & \simeq \LB_{\alpha,u} \cdot F_{uw}. %
\end{align*}
Since $\LB_{\alpha,u}  \cdot F_{uw} $ is an affine space \cite{Brion04},
we obtain \eqref{Oitem6}.
\end{proof}

\subsection{The convolution algebra}

For a given quasi-projective variety $X$ over $\C$ with an action of
a complex linear algebraic group $G$, let $\D^b_G(X)$ be the bounded
$G$-equivariant derived category of sheaves of $\C$-vector spaces on
$X$. Let $\C_X$ be the constant sheaf on $X$ and let $\varpi_X \in
\D^b_G(X)$ be the $G$-equivariant dualizing complex. For any complex
$\mathcal{L} \in \D^b_G(X)$, let $H_G^k(X, \mathcal{L}) =
\Hom_{\D^b_{\LG}(X)}(\C_X, \mathcal{L} [k])$, where $[ \cdot ]$ is
the translation functor.
For $k \in \Z$, define
$$ H_k^G(X) = H_G^k(X, \varpi_X), \qquad  H^G_*(X) = \soplus_{k\in \Z} H^G_k(X) $$
and
$$  H^k_G(X) = H_G^k(X, \C_X),\qquad H_G^*(X) = \soplus_{k\in \Z} H_G^k(X). $$
Note that, if $X \hookrightarrow M$ is a closed embedding into a
smooth variety $M$, then we have $ H_k^G(X) = H_G^{k+2\dim_{\C}M}(M,
M \setminus X)$.
The graded module $ H_G^*(X)$ has a graded ring structure
and $H^G_*(X)$ has a graded $ H_G^*(X)$-module structure.
We refer to \cite{BL94,Brion98,CG97}
for more details on equivariant cohomologies.

We now return to our setting. Let $\gH$ be the ring given by
$$
\gH = H_{\LH^\Omega}^*( {\rm pt}) \simeq \bigotimes_{a\in \Omega} \C[ \hbar_a ],
$$
where $\hbar_a\in H_{\LH^\Omega}^*( {\rm pt}) $ corresponds to
the one-dimensional representation of
the $a$-th factor $\C^*$ of $\LH^\Omega$.
This will play the role of the base ring given in $\eqref{Eq: def of bH}$.
Let
$\gS$ be the ring defined by
\begin{align} \label{Eq: def of S}
\gS = H_{\LG_\alpha^\Omega}^*( {\rm pt}) \simeq \gH \otimes H_{\LG_\alpha}^*( \rm pt).
\end{align}

For $\nu, \nu' \in I^\alpha$, we define
\begin{equation} \label{Eq: def of gR and gPol}
\begin{aligned}
& \gR_{\nu,\nu'} \seteq H_*^{\LG_\alpha^\Omega}( \vZ_{\nu,\nu'} ) \langle -2\dim_\C \tF_{\nu} \rangle,  \ \ \  \gR(\alpha) \seteq\soplus_{\nu,\nu' \in I^\alpha} \gR_{\nu,\nu'}, \\
&\gPol_{\nu} \seteq H^*_{\LG_\alpha^\Omega}(\tF_\nu),   \quad \qquad \qquad  \qquad \quad \gPol(\alpha) \seteq H^*_{\LG_\alpha^\Omega}(\tF_\alpha)\simeq
\soplus_{\nu \in I^\alpha} \gPol_{\nu}.
\end{aligned}
\end{equation}

For $k = 1,\ldots,m$, let $\chi_k \in \Lt^*$ be the weight of the $\Lt$-module induced from the $\LT$-module $\mathbb{F}_k$
given in $\eqref{Eq: def of 1 dim F}$ by differentiation.
Then we have
\eq
&&\ H_{\LT^\Omega}^*(\mathrm{pt})
\simeq \gH \otimes H_{\LT}^{*}(\mathrm{pt})
\simeq  \gH [\chi_1, \ldots, \chi_m] .
\eneq

For $\nu\in I^\alpha$, we denote by $\mathcal{L}_{\nu,k}$
the $G_\alpha^\Omega$-equivariant line bundle on $\F_\nu$ 
%
%
assigning $F_k/F_{k-1}$ to a flag $0=F_0\subset F_1\subset\cdots\subset F_m=V$
in $\F_\nu$.
Let $\chi_k(\nu)$ be
the $\LG_\alpha^\Omega$-equivariant first Chern class
$c_1(\mathrm{p}_\nu^*\mathcal{L}_{ \nu, k})\in H_{\LG_\alpha^\Omega}^2(\tF_\nu )$.
For any $w\in  \weyl$ such that $\nu=\nu_w$,
we have isomorphisms
$$H_{\LG_\alpha^\Omega}^*( \tF_\nu )\simeq
H_{\LG_\alpha^\Omega}^*( \F_\nu )\simeq
H_{\LG_\alpha^\Omega}^*
\bl\LG_\alpha^\Omega/(\LB_{\alpha,w}\times \LH^\Omega)\br
\simeq H^*_{\LT^\Omega}(\mathrm{pt}).$$
 Moreover we see that
\eq
&&\ \left\{\ \parbox{70ex}
{the composition of these isomorphisms coincides with the composition
$$H_{\LG_\alpha^\Omega}^*( \tF_\nu )\to
H_{\LT^\Omega}^*( \tF_\nu)
\to  H_{\LT^\Omega}^*( \{(0, F_w)\})\simeq  H_{\LT^\Omega}^*( {\rm pt})
\simeq\gH[\chi_1, \ldots, \chi_m],$$
 and it sends
$\chi_k(\nu)$ to $\chi_{w(k)}$.}\right.\label{eq:locchi}
\eneq
The last assertion follows from the fact that
the stalk of $\mathrm{p}_\nu^*\mathcal{L}_{\nu, k}$ at $(0,F_w)\in\tF_\nu$
is isomorphic to $\mathbb{F}_{w(k)}$ as a $\LT^\Omega$-module.

Therefore, we have
\eq&&\gPol_{\nu} =H_{\LG_\alpha^\Omega}^*( \tF_\nu )\simeq  \gH
[\chi_1(\nu), \ldots, \chi_m(\nu)]
\label{Eq: isom between flags}
\eneq
as an algebra, and
\begin{equation*}
\begin{aligned}
\gPol(\alpha) & = \soplus_{\nu \in
I^\alpha}  H_{\LG_\alpha^\Omega}^*( \tF_\nu ) \simeq \soplus_{\nu
\in I^\alpha} \gH[\chi_1(\nu), \ldots, \chi_m(\nu)].
\end{aligned}
\end{equation*}

For $ \nu,\nu',\nu'' \in I^\alpha$, let $\mathfrak{p}_{\nu,\nu'}$ (resp.\ $\mathfrak{p}_{\nu',\nu''}$, $\mathfrak{p}_{\nu,\nu''}$) be
the natural projection from $\tF_{\nu} \times \tF_{\nu'} \times \tF_{\nu''} $ to $\tF_{\nu} \times \tF_{\nu'}$
(resp.\ $\tF_{\nu'} \times \tF_{\nu''}$, $\tF_{\nu} \times \tF_{\nu''}$). Since
$\mathfrak{p}_{\nu,\nu'}^{-1}(\vZ_{\nu,\nu'}) \cap \mathfrak{p}_{\nu',\nu''}^{-1}(\vZ_{\nu',\nu''}) \simeq \tF_{\nu} \times_{\E_\alpha} \tF_{\nu'} \times_{\E_\alpha} \tF_{\nu''}$,
the restriction of the natural projection
$$\mathfrak{p}_{\nu,\nu''}: \mathfrak{p}_{\nu,\nu'}^{-1}(\vZ_{\nu,\nu'}) \cap \mathfrak{p}_{\nu',\nu''}^{-1}(\vZ_{\nu',\nu''}) \longrightarrow \vZ_{\nu,\nu''}$$
is proper. Hence, for $c_1 \in
H_p^{\LG_\alpha^\Omega}(\vZ_{\nu,\nu'})$ and $ c_2 \in
H_q^{\LG_\alpha^\Omega}(\vZ_{\nu',\nu''})$, the {\it convolution
product}
$$\star : H_p^{\LG_\alpha^\Omega}(\vZ_{\nu,\nu'}) \times
H_{q}^{\LG_\alpha^\Omega}(\vZ_{\nu',\nu''}) \rightarrow
H_{p+q+2\dim_\C \tF_{\nu'}}^{\LG_\alpha^\Omega}(\vZ_{\nu,\nu''})$$ is
defined by
\begin{equation} \label{eq:convolution}
c_1 \star c_2 =
(\mathfrak{p}_{\nu,\nu''})_*( \mathfrak{p}_{\nu,\nu'}^{*}(c_1) \cap
\mathfrak{p}_{\nu',\nu''}^{*}(c_2) ),
\end{equation}
where $\cap$ is
the intersection pairing \cite{CG97}. Then $\gR(\alpha)$ becomes an
associative graded $\gH$-algebra with the convolution product
$\star$. Moreover $\gPol(\alpha)$ has the graded $\gR(\alpha)$-module
structure arising from the convolution product:
\eqn
H_{p}^{\LG_\alpha^\Omega}(\vZ_{\nu,\nu'}) \otimes
H^{q}_{\LG_\alpha^\Omega}(\tF_{\nu'})
&\to&
H_{p}^{\LG_\alpha^\Omega}(\vZ_{\nu,\nu'}) \otimes
H^{q}_{\LG_\alpha^\Omega}(\vZ_{\nu,\nu'})\\
&\to& H_{p+q}^{\LG_\alpha^\Omega}(\vZ_{\nu,\nu'})
\to H_{p+q}^{\LG_\alpha^\Omega}(\tF_{\nu})
\simeq H^{p+q+2\dim_{\C}\tF_\nu}_{\LG_\alpha^\Omega}(\tF_{\nu}).
\eneqn
Here, the first arrow is the pull-back by the projection
$\vZ_{\nu,\nu'}\to \tF_{\nu'}$, the second arrow is the multiplication
and the last
arrow is the push-forward by the proper map
$\vZ_{\nu,\nu'}\to\tF_{\nu}$.

Recall the maps $\mathrm{q}_{\nu, \nu'}:\vZ_{\nu,\nu'} \rightarrow \F_{\nu} \times \F_{\nu'}$ and
$\mathrm{q}_{\alpha}:\vZ_\alpha \rightarrow \F_\alpha \times \F_\alpha$
 given in $\eqref{Eq: map q1}$ and $\eqref{Eq: map q2}$.
For $w\in \weyl$ and $\nu, \nu' \in I^\alpha $, we define
\begin{equation} \label{Eq: def of Zw}
\begin{aligned}
& \vZ^w_\alpha = \overline{\mathrm{q}_{\alpha}^{-1}(\vO^w_\alpha)}, \ \qquad \vZ^{\le w}_\alpha = \bigcup_{w' \le w} \vZ^{w'}_\alpha, \\
& \vZ^w_{\nu,\nu'} = \vZ^w_\alpha \cap \vZ_{\nu,\nu'}, \quad \vZ^{\le w}_{\nu,\nu'} = \vZ^{\le w}_\alpha \cap \vZ_{\nu,\nu'},
\end{aligned}
\end{equation}
where $\overline{\mathrm{q}_{\alpha}^{-1}(\vO^w_\alpha)}$ is the closure in $\vZ_\alpha$.
Note that $\vZ^{\le w}_\alpha = \bigcup_{w' \le w} \mathrm{q}_{\alpha}^{-1}(\vO^{w'}_\alpha)$ is a closed $\LG_\alpha^\Omega$-subvariety.
\begin{Lem}\label{lem:Zsj}
For $\nu,\nu'\in I^\alpha$ $w\in \weyl$ and $j=1,\ldots, m-1$, we have
\bnum
\item
If $\nu'\not=w^{-1}\nu$, then $\vZ^w_{\nu,\nu'}=\emptyset$.
If $\nu'=w^{-1}\nu$, then $\vZ^w_{\nu,\nu'}$ is an irreducible variety.
\item
\parbox[t]{65ex}{
$\vZ^{s_j}_{\alpha}=\{(x,F,F')\in \E_\alpha\times\F_\alpha\times\F_\alpha$\\
$\hs{17ex}\mid \text{$F_k=F'_k$ for $k\not=j$, $xF_k\subset F_{k-1}$ for
$k\not=j,j+1$ and $xF_{j+1}\subset F_{j-1}$}\}$.
}\\[.5ex]
\item
$\vZ^{s_j}_{\alpha}$ is a smooth variety.
\ee
\end{Lem}
\begin{proof}
(i) follows immediately from Lemma~\ref{Lem: properties of O}
\eqref{Oitem2} and the fact that
$\mathrm{q}_{\alpha}^{-1}(\vO^w_\alpha)\to \vO^w_\alpha$ is a vector bundle.

Let us show (ii). The right-hand side is a vector bundle over
$\oO^{s_j}_{\alpha}$ by Lemma~\ref{Lem: properties of O} \eqref{Oitem3}. Since
$F_j\cap F'_j=F_{j-1}$ for $(F,F')\in \vO^{s_j}_{\alpha}$,
$\mathrm{q}_\alpha^{-1}\vO^{s_j}_{\alpha}$ is an open dense subset
of the right-hand side of (ii).

\noi
(iii) follows immediately from (ii).
\end{proof}

For $w\in \weyl$ and $\nu, \nu' \in I^\alpha $, let
\begin{align} \label{Eq: def of Rw}
  \gR^{\le w}_{\nu,\nu'} = H_*^{\LG_\alpha^\Omega}( \vZ^{\le w}_{\nu,\nu'} )
\langle -2\dim_\C \tF_{\nu} \rangle, \qquad
\gR^{\le w}_{\alpha} =  \soplus_{\nu,\nu' \in I^\alpha} \gR^{\le w}_{\nu,\nu'}.
\end{align}
For $w \in \weyl$,
we denote by $[\vZ^{ w}_\alpha] =\sum_{\nu,\nu'}
[\vZ^{w}_{\nu,\nu'}] $ the class in $\gR^{\le w}_\alpha$ arising
from $\vZ^{ w}_\alpha = \bigsqcup_{\nu,\nu'} \vZ^{w}_{\nu,\nu'} $.
The convolution $\star$ gives a left $\gR^{\le \id}_{\alpha}$-module
structure on the space $\gR^{\le w}_{\alpha}$.
Using $ \vZ^{\le\id}_{\nu,\nu} \simeq \tF_{\nu}$, the isomorphism
$\eqref{Eq: isom between flags}$ yields
\begin{equation} \label{Eq: isom of cohomology}
\begin{aligned}
H^{\LG_\alpha^\Omega}_*( \vZ^{\le\id}_{\nu,\nu} )\langle -2\dim_\C \tF_{\nu} \rangle
\simeq H_{\LG_\alpha^\Omega}^*( \vZ^{\le\id}_{\nu,\nu} ) \simeq H_{\LG_\alpha^\Omega}^*( \tF_\nu )
\simeq  \gH[\chi_1(\nu), \ldots, \chi_m(\nu)] .
\end{aligned}
\end{equation}
Let $\varkappa_k(\nu) \in H_{\LG_\alpha^\Omega}^*(  \vZ^{\le\id}_{\nu,\nu} )$ be the
image of $\chi_k(\nu)$ under the isomorphism $\eqref{Eq: isom of
cohomology}$ for $k=1,\ldots,m$. Then  
we have
\begin{equation} \label{Eq: H(Ze)}
\begin{aligned}
\gR^{\le \id}_{\alpha} 
\simeq \soplus_{\nu \in I^\alpha} H_{\LG_\alpha^\Omega}^*( \vZ^{\le\id}_{\nu,\nu})
\simeq \soplus_{\nu \in I^\alpha} \gH[ \varkappa_1(\nu), \ldots, \varkappa_m(\nu)]
\end{aligned}
\end{equation}
as an algebra.

 The following lemma will play an important role in proving our
main result.

\begin{Lem}  \label{Lem: inj of R} \
\bnum
\item The closed embedding $ \vZ^{\le w}_\alpha \hookrightarrow \vZ_\alpha $ induces an injective graded $\gR^{\le \id}_{\alpha}$-bimodule homomorphism $\gR^{\le w}_{\alpha} \hookrightarrow \gR(\alpha)$.\label{coh 1}
\item For $w \in \weyl$, $\gR^{\le w}_{\alpha}$ is a free $\gR_\alpha^{\le \id}$-module of rank
$\#\{ w' \in \weyl \mid w' \le w \}$
and
$$\gR^{\le w}_{\alpha} = \soplus_{w'\le w} \gR^{\le \id}_{\alpha} \star [\vZ^{w'}_\alpha].$$
\item If  $\ell(s_jw) = \ell(w) + 1$, we have
$$[\vZ_\alpha^{s_j}] \star [\vZ_\alpha^{w}] = [\vZ_\alpha^{s_jw}] \quad
\text{in $\gR^{\le s_jw}_{\alpha}/\gR^{< s_jw}_{\alpha}$,} $$
where $\gR^{< u}_{\alpha} \seteq \sum_{u' < u} \gR^{\le u'}_\alpha $ for $u\in \weyl$.
\end{enumerate}
\end{Lem}
\begin{proof}
Since (i) is a consequence of (ii), we first prove (ii). Let $\ell = \#\{ w' \in \weyl \mid w' \le w \}$.
We give a total order $\prec$ on $\{ w' \in \weyl \mid w' \le w  \}$ by
$$ \id = w_1 \prec \cdots \prec w_k \prec w_{k+1}\prec \cdots \prec w_{\ell } = w  $$
such that $\ell(w_{k+1}) = \ell(w_k)$ or $\ell(w_{k+1}) = \ell(w_k)+1$ for all $k$.
Set
$$\vZ^{\preceq k } = \bigcup_{i \le k} \vZ_\alpha^{w_i} \ \text{ for } k=  1 ,\ldots, \ell.$$
Then $\vZ^{\preceq k }$ is a closed subset of $\vZ_\alpha$.
Since $ \vZ_\alpha^{\le w} = \vZ^{\preceq \ell}$, it suffices to show that
\begin{align*}
H_{*}^{\LG_\alpha^\Omega}( \vZ^{\preceq k} ) = \bigoplus_{k' \le k} \gR^{\le \id}_{\alpha} \star [\vZ^{w_{k'}}_\alpha], \qquad
H_{2t+1}^{\LG_\alpha^\Omega}( \vZ^{\preceq k} ) = 0
\quad\text{for all $k=1,\ldots,\ell$.}
\end{align*}

We use induction on $k$. Assume $k=1$.
Since $\deg(\varkappa_k(\nu))=2$ for $k=1,\ldots,m$ and $\nu \in I^\alpha$,
$\eqref{Eq: H(Ze)}$ implies
\begin{align*}
H_{*}^{\LG_\alpha^\Omega}( \vZ^{\preceq 1 } ) = \gR^{\le \id}_{\alpha} \star [\vZ^{\id}_\alpha], \qquad
H_{2t+1}^{\LG_\alpha^\Omega}( \vZ^{\preceq 1 } ) = 0\quad\text{for $t\in \Z$.}
\end{align*}
 Suppose $k>1$. By the induction hypothesis, we have
\begin{align} \label{Eq: hypothesis1}
H_{*}^{\LG_\alpha^\Omega}( \vZ^{\preceq k-1} ) = \bigoplus_{k' \le k-1} \gR^{\le \id}_{\alpha} \star [\vZ^{w_{k'}}_\alpha], \qquad
H_{2t+1}^{\LG_\alpha^\Omega}( \vZ^{\preceq k-1} ) = 0\quad\text{for $t\in \Z$.}
\end{align}

Let $\mathrm{q}_k = \mathrm{q}_\alpha |_{\vZ^{\preceq k}}: \vZ^{\preceq k} \rightarrow \F_\alpha \times \F_\alpha$ and consider the following diagram:
$$
\xymatrix{
 {\cdots\hs{1ex} }\ar@{^(->}[r]  &   \vZ^{\preceq k-1} \ar[dr]_{\mathrm{q}_{k-1}}
\ar@{^(->}[r] &  \vZ^{\preceq k} \ar[d]_{\mathrm{q}_{k}} \ar@{^(->}[r] &
\vZ^{\preceq k+1}  \ar[dl]^{\mathrm{q}_{k+1}} \ar@{^(->}[r] & \hs{1ex} \cdots \\
& & \F_\alpha \times \F_\alpha \ar[d]^{\mathrm{pr}_1} & & \\
& & \F_\alpha  & &
}.
$$
Here, $\mathrm{pr}_1: \F_\alpha \times \F_\alpha \rightarrow \F_\alpha$ is the natural projection onto the first factor.
Since $ \vO^{w_k}_\alpha \subset \im(\mathrm{q}_k) $,
we have the surjective maps given below
$$
\xymatrix{
\mathrm{q}_k^{-1}( \vO^{w_k}_{w', w' w_k}) \ar@{->>}[r]^{\ \  \mathrm{q}_k}& \vO^{w_k}_{w',w' w_k} \ar@{->>}[r]^{  \mathrm{pr}_1}& \F_{{\nu}_{w'}}
}\quad\text{for $w'\in \weyl_\alpha \backslash \weyl$.}
$$
By Lemma \ref{Lem: properties of O} \eqref{Oitem6},
the map $\mathrm{q}_k^{-1}( \vO^{w_k}_{w', w' w_k})\to \F_{{\nu}_{w'}}$
is locally trivial with contractible fiber, and hence
we have that
\begin{equation}\label{Eq: Thom iso 1}
\begin{aligned}
 H^{\LG_{\alpha}^\Omega}_*( \mathrm{q}_k^{-1}( \vO^{w_k}_{w',w'w_k}) )
&\simeq H_{\LG_{\alpha}^\Omega}^*( \F_{\nu_{w'}})
\langle  2\dim_\C \vZ_{{\nu}_{w'}, {\nu}_{w'w_k}}^{w_k} \rangle\\
&\simeq\gH[\chi_1(\nu_{w'}), \ldots, \chi_m(\nu_{w'})]
\langle  2\dim_\C \vZ_{{\nu}_{w'}, {\nu}_{w'w_k}}^{w_k} \rangle.
\end{aligned}
\end{equation}
Note that $\mathrm{q}_k^{-1}( \vO^{w_k}_{w', w' w_k})$
is a dense open smooth subset of
$\vZ_{{\nu}_{w'}, {\nu}_{w'w_k}}^{w_k}$. Since $ \vZ^{\preceq
k}\setminus \vZ^{\preceq k-1}  = \mathrm{q}_k^{-1}(
\vO^{w_k}_\alpha)$, we get a long exact sequence (\cite[(2.6.10)]{CG97})
\begin{equation*}
\begin{aligned}
\xymatrix@C=2.7ex{
\cdots  \ar[r] & H^{\LG_{\alpha}^\Omega}_i( \vZ^{\preceq k-1} ) \ar[r] & H^{\LG_{\alpha}^\Omega}_i( \vZ^{\preceq k} ) \ar[r] & H^{\LG_{\alpha}^\Omega}_i( \mathrm{q}_k^{-1}( \vO^{w_k}_\alpha)  )
\ar[r] & H^{\LG_{\alpha}^\Omega}_{i+1}( \vZ^{\preceq k-1} )
\ar[r] & \cdots
}
\end{aligned}
\end{equation*}
for all $i\in \Z$.
Since $H^{\LG_{\alpha}^\Omega}_{2t+1}( \vZ^{\preceq k-1} )=
H^{\LG_{\alpha}^\Omega}_{2t+1}(\mathrm{q}_k^{-1}( \vO^{w_k}_\alpha))=0$ for any $t\in\Z$
 by $\eqref{Eq: hypothesis1}$ and $\eqref{Eq: Thom iso 1}$,
the following sequence is exact:
\begin{align} \label{Eq: exact seq for Z}
\xymatrix{
0  \ar[r] & H^{\LG_{\alpha}^\Omega}_*( \vZ^{\preceq k-1} ) \ar[r] & H^{\LG_{\alpha}^\Omega}_*(\vZ^{\preceq k} ) \ar[r] & H^{\LG_{\alpha}^\Omega}_*( \mathrm{q}_k^{-1}( \vO^{w_k}_\alpha)  )
\ar[r] &0
}.
\end{align}
By the isomorphisms $\eqref{Eq: H(Ze)}$ and  $\eqref{Eq: Thom
iso 1}$, we conclude that $H^{\LG_{\alpha}^\Omega}_*( \mathrm{q}_k^{-1}(
\vO^{w_k}_\alpha))$ is a free $\gR^{\le \id}_{\alpha}$-module
generated by the image of $[\vZ^{w_k}_\alpha]\in
H_{*}^{\LG_\alpha^\Omega}( \vZ^{\preceq k} )$. Hence the short
exact sequence $\eqref{Eq: exact seq for Z}$ splits and
\begin{align*}
H_{*}^{\LG_\alpha^\Omega}( \vZ^{\preceq k} ) =
H_{*}^{\LG_\alpha^\Omega}( \vZ^{\preceq k-1} ) \oplus
\gR^{\le \id}_{\alpha} \star [\vZ^{w_{k}}_\alpha], \qquad
H_{2t+1}^{\LG_\alpha^\Omega}( \vZ^{\preceq k} ) = 0.
\end{align*}
Hence the induction proceeds by \eqref{Eq: hypothesis1}
and we complete the proof of (ii).

We now proceed to prove the assertion (iii). For $k, k' = 1,2,3$, let $ \mathbf{p}_{k, k'}$
$:\tF_\alpha \times \tF_\alpha \times \tF_\alpha \rightarrow \tF_\alpha \times \tF_\alpha$ be the projection given by $\bl (x_1,F_1), (x_2,F_2), (x_3,F_3) \br\mapsto
\bl (x_{k},F_k), (x_{k'},F_{k'}) \br $. Combining the assertion (ii) with the inclusion
$$  \mathbf{p}_{12}^{-1}  (\vZ^{s_j}_\alpha) \cap  \mathbf{p}_{23}^{-1} (\vZ^{w}_\alpha)
\subset  \mathbf{p}_{12}^{-1} \mathrm{q}^{-1}_{\alpha} (\overline{\vO}^{s_j}_{\alpha}) \cap
\mathbf{p}_{23}^{-1} \mathrm{q}^{-1}_{\alpha} (\overline{\vO}^{w}_{\alpha})
\subset \mathbf{p}_{13}^{-1} \mathrm{q}^{-1}_{\alpha} (\overline{\vO}^{s_jw}_{\alpha})
\subset \mathbf{p}_{13}^{-1} (\vZ^{\le s_jw}_\alpha), $$
we have
\begin{align} \label{Eq: convolution of two elts}
[\vZ^{s_j}_\alpha] \star [\vZ^{w}_\alpha] = c \star [\vZ^{s_jw}_\alpha] + \sum_{w' < s_jw}c_{w'} \star [\vZ^{w'}_\alpha]
\end{align}
for $c, c_{w'} \in \gR^{\le \id}_{\alpha}$. To prove (3), it suffices to show $c=1$.

\medskip
For $k, k' = 1,2,3$, let $ \check{\mathbf{p}}_{k, k'} $
$:\F_\alpha \times \F_\alpha \times \F_\alpha \rightarrow \F_\alpha \times \F_\alpha$ be the natural projection given by
$\left( F_1, F_2, F_3 \right) \mapsto \left( F_{k},F_{k'} \right) $.
Then it is known that $\ell(s_jw) = \ell(w) + 1$ implies the following consequences:
\eq
&&\hs{5ex}\left\{
\parbox[c]{80ex}{
\be[(a)]
\item
$\check{\mathbf{p}}_{12}^{-1} (\overline{\vO}^{s_j}_{u,us_j})\cap
 \check{\mathbf{p}}_{23}^{-1} (\overline{\vO}_{us_j,us_jw}^w)\cap
\check{\mathbf{p}}_{13}^{-1} (\vO^{s_jw}_{u,us_jw})
=\check{\mathbf{p}}_{12}^{-1} (\vO^{s_j}_{u,us_j})
\cap\check{\mathbf{p}}_{23}^{-1} (\vO^{w}_{us_j,us_jw})$,\label{iten:Ows1}

\vs{1ex}
\item $ \check{\mathbf{p}}_{12}^{-1} (\vO^{s_j}_{u,us_j})$
 and $\check{\mathbf{p}}_{23}^{-1} (\vO^{w}_{us_j,us_jw})$
intersect transversally,

\vs{1ex}
\item
The projection $\check{\mathbf{p}}_{13}$ induces an isomorphism
$$ \check{\mathbf{p}}_{12}^{-1} (\vO^{s_j}_{u,us_j})
\cap\check{\mathbf{p}}_{23}^{-1} (\vO^{w}_{us_j,us_jw})
\isoto \vO^{s_jw}_{u,us_jw}.$$
\ee}\right.
\label{eq:Ows}
\eneq

Let us set
$$\overset{\circ}{\vZ}{^{u}_{u',u''}} = \mathrm{q}^{-1}_{\alpha}(\vO^{u}_{u',u''})\ \text{ and }\
\overset{\circ}{\vZ}{^{u}_{\alpha}} = \mathrm{q}^{-1}_{\alpha}(\vO^{u}_{\alpha})
\quad\text{for $u,u',u''\in \weyl$. }$$

Then, \eqref{eq:Ows} \eqref{iten:Ows1} implies
\begin{align} \label{Eq: intersection}
\mathbf{p}_{12}^{-1} & (\vZ^{s_j}_{\alpha})  \cap
\mathbf{p}_{23}^{-1} (\vZ^{w}_{\alpha}) \cap \mathbf{p}_{13}^{-1}
(\overset{\circ}{\vZ}{^{s_jw}_{\alpha}}) = \mathbf{p}_{12}^{-1}
(\overset{\circ}{\vZ}{^{s_j}_{\alpha}})  \cap \mathbf{p}_{23}^{-1}
(\overset{\circ}{\vZ}{^w_{\alpha}}).
\end{align}
Note that $ \mathbf{p}_{12}^{-1}   (\overset{\circ}{\vZ}{^{s_j}_{\alpha}})$ (resp.\
$ \mathbf{p}_{23}^{-1} (\overset{\circ}{\vZ}{^w_{\alpha}})$) is a smooth dense open subset of
$ \mathbf{p}_{12}^{-1} (\vZ^{s_j}_{\alpha})$ (resp.\ $ \mathbf{p}_{23}^{-1} (\vZ^{w}_{\alpha})$).
Hence, by $\eqref{Eq: exact seq for Z}$, (iii) is reduced to
\eq&&\left\{
\parbox{74ex}{
\be[(a)]
\item
$\mathbf{p}_{13}$ induces an isomorphism
$ \mathbf{p}_{12}^{-1} (\overset{\circ}{\vZ}{^{s_j}_{u,us_j}}) \cap
 \mathbf{p}_{23}^{-1}(\overset{\circ}{\vZ}{^w_{us_j,us_jw}})\isoto
\overset{\circ}{\vZ}{^{s_jw}_{u,us_jw}}$.\label{Sitem1}
\item
 $ \mathbf{p}_{12}^{-1}  (\overset{\circ}{\vZ}{^{s_j}_{u,us_j}})$ and $ \mathbf{p}_{23}^{-1} (\overset{\circ}{\vZ}{^{w}_{us_j,us_jw}})$
intersect transversally.\label{Zitem2}
\ee
}\right.\label{eq:SZ}
\eneq
Let us show first \eqref{eq:SZ} \eqref{Sitem1}. By using \eqref{eq:Ows},
it is enough to show that
if $F_1$ and $F_3$ are strictly $x$-stable,
then $F_2$ is strictly $x$ stable
for $x\in\E_\alpha$ and $(F_1, F_2, F_3)\in
\check{\mathbf{p}}_{12}^{-1} (\vO^{s_j}_{u,us_j})
\cap\check{\mathbf{p}}_{23}^{-1} (\vO^{w}_{us_j,us_jw})$.
Since $\check{\mathbf{p}}_{12}^{-1} (\vO^{s_j}_{u,us_j})
\cap\check{\mathbf{p}}_{23}^{-1} (\vO^{w}_{us_j,us_jw})\simeq\vO^{s_jw}_{u,us_jw}$
is a $G_\alpha$-orbit, we may assume that
$(F_1, F_2, F_3)=(F_u,F_{us_j},F_{us_jw})$.

Since $x$ is nilpotent, it is enough to show that $x(F_2)_k\subset (F_2)_k$
for all $k$.
We have $(F_2)_k=(F_1)_{k}$ for $k\not=j,\,j+1$.
It is therefore obvious that $x(F_2)_k\subset (F_2)_{k}$ for $k\not=j,\,j+1$.
We have
$(F_2)_j=(F_2)_{j-1}+\bF_{u(j+1)}$, which implies that
$x(F_2)_j\subset x(F_1)_{j-1}+x\bF_{u(j+1)}$.
Since $\bF_{u(j+1)}\subset (F_1)_{j+1}\cap(F_{us_jw})_{w^{-1}(j)}$, we have
$x\bF_{u(j+1)}\subset (F_1)_j\cap(F_{us_jw})_{w^{-1}(j)}$.
On the other hand, we have $(F_1)_j=(F_1)_{j-1}+\bF_{uj}$
and $\bF_{uj}\cap (F_{us_jw})_{w^{-1}(j)}=\bF_{us_jw(w^{-1}(j+1))}\cap (F_{uws_j})_{w^{-1}(j)}=0$
because $w^{-1}(j+1)>w^{-1}(j)$ by the assumption $s_jw>w$.
It implies that $(F_1)_j\cap(F_{us_jw})_{w^{-1}(j)}\subset (F_1)_{j-1}$.
Hence we have $x(F_2)_j\subset (F_2)_{j-1}$.

We have
$(F_2)_{j+1}=(F_2)_{j}+\bF_{u(j)}\subset(F_2)_{j}+(F_1)_j$.
Hence $x(F_2)_{j+1}\subset x(F_2)_{j}+x(F_1)_j
\subset (F_2)_{j-1}$.
Hence we complete the proof of \eqref{eq:SZ} \eqref{Sitem1}.

\smallskip
Let us show \eqref{eq:SZ} \eqref{Zitem2}.
Since $ \check{\mathbf{p}}_{12}^{-1} (\vO^{s_j}_{u,us_j})$
 and $\check{\mathbf{p}}_{23}^{-1} (\vO^{w}_{us_j,us_jw})$
intersect transversally, it is enough to show that
the fiber of
$ \mathbf{p}_{1,2}^{-1} (\overset{\circ}{\vZ}{^{s_j}_{u,us_j}}) \rightarrow \check{\mathbf{p}}_{1,2}^{-1} (\vO^{s_j}_{u,us_j}) $
and the fiber of
 $  \mathbf{p}_{2,3}^{-1} (\overset{\circ}{\vZ}{^{w}_{us_j,us_jw}}) \rightarrow
 \check{\mathbf{p}}_{2,3}^{-1} (\vO^{w}_{us_j,us_jw}) $
intersect transversally.

Let $\mathcal{T}^{1,2}$ (resp.\ $\mathcal{T}^{2,3}$) be the fiber of
$ \mathbf{p}_{1,2}^{-1} (\overset{\circ}{\vZ}{^{s_j}_{u,us_j}}) \rightarrow \check{\mathbf{p}}_{1,2}^{-1} (\vO^{s_j}_{u,us_j}) $
(resp.\ $  \mathbf{p}_{2,3}^{-1} (\overset{\circ}{\vZ}{^{w}_{us_j,us_jw}}) \rightarrow
 \check{\mathbf{p}}_{2,3}^{-1} (\vO^{w}_{us_j,us_jw}) $)
at $(F_u,F_{us_j},F_{us_jw})$.
We will show that
$\mathcal{T}^{1,2}$ and $\mathcal{T}^{2,3}$ intersect transversally.
Recall that $\ve_v$ is the fiber of $\tF_\alpha\to\F_\alpha$
at $F_v$ for $v\in \weyl$ (see \eqref{Eq: def of ew}).
Set $E_1=\ve_u$, $E_2=\ve_{us_j}$ and $E_3=\ve_{us_jw}$.
Then we have
\eqn
\mathcal{T}^{1,2}
&=&\{(v_1,v_2,v_3)\in E_1\oplus E_2\oplus E_3\mid
 v_1=v_2\},\\
\mathcal{T}^{2,3}
&=&\{(v_1,v_2,v_3)\in E_1\oplus E_2\oplus E_3\mid
 v_2=v_3\}.
\eneqn
In order to see that $\mathcal{T}^{1,2}$ and
$\mathcal{T}^{2,3}$ intersect transversally in $E_1\oplus E_2\oplus E_3$,
it is enough to show that
\eq &&E_2=(E_1\cap E_2)+(E_2\cap E_3).\label{eq:E3}
\eneq
Set $A(w)=\{(k,k')\mid
\text{$1\le k,k'\le m$ and $w^{-1}(k)>w^{-1}(k')$}\}$,
and $A_1=A(u)$, $A_2=A(us_j)$ and $A_3=A(us_jw)$.
Then we have
$E_s=\soplus_{(k,k')\in A_s}
\mathrm{Hom}(\bF_k,\bF_{k'})^{\Omega_{\nu^\circ_k,\nu^\circ_{k'}}}$ for $s=1,2,3$.
Hence we have reduced \eqref{eq:E3} to
$$A_2\subset A_1\cup A_3,$$
which immediately follows from
$A_2\setminus A_1=\{(u(j),u(j+1))\}$
and $(u(j),u(j+1))\in A_3$. Here the last statement is a consequence of
$w^{-1}(j+1)>w^{-1}(j)$.
This completes the proof of \eqref{eq:SZ} \eqref{Zitem2}.
\end{proof}

We now choose a set of generators of the convolution algebra
$\gR(\alpha)$. By Lemma \ref{Lem: inj of R} and $\eqref{Eq: H(Ze)}$,
we have an injective homomorphism
\begin{align} \label{Eq: injection - gen}
\xymatrix@C=4ex{
\gR^{\le \id}_{\alpha} \simeq \bigoplus_{\nu \in I^\alpha} \gH[ \varkappa_1(\nu), \ldots, \varkappa_m(\nu) ] \   \ar@{^(->}[r] &  \gR(\alpha).
}
\end{align}

For $\nu \in I^\alpha$, we define
\begin{equation} \label{Eq: def of e and x}
\begin{aligned}
e(\nu) &=[\vZ^{e}_{\nu,\nu}]= \text{the image of 1 in $\gH[ \varkappa_1(\nu), \ldots, \varkappa_m(\nu) ]$ under $\eqref{Eq: injection - gen}$,} \\
\end{aligned}
\end{equation}
For $\nu \in I^\alpha$ and $j=1,\ldots,m-1$, let $\tilde{\tau}_{j} =
[\vZ^{s_j}_\alpha] \in \gR(\alpha)$ and set
\begin{equation} \label{Eq: def of tau}
\begin{aligned}
\tau_j(\nu) =[\vZ^{s_j}_{s_j\nu,\nu}]=e(s_j\nu) \star \tilde{\tau}_{j} \star e(\nu)
=e(s_j\nu) \star \tilde{\tau}_{j} =\tilde{\tau}_{j} \star e(\nu).
\end{aligned}
\end{equation}

\begin{Lem} \label{Lem: degree of generators}
For $\nu=(\nu_1, \ldots, \nu_m) \in I^\alpha$, we have
\begin{align*}
e(\nu) \in H^{\LG_\alpha^\Omega}_{\mathrm{d}_\nu}(\vZ_{\nu,\nu}),
\quad \varkappa_i(\nu) \in H^{\LG_\alpha^\Omega}_{\mathrm{d}_\nu+2}(\vZ_{\nu,\nu}),  \quad
\tau_j(\nu) \in H^{\LG_\alpha^\Omega}_{\mathrm{d}_{s_j\nu} + \mathrm{t}_{\nu,j}}(\vZ_{s_j\nu,\nu}),
\end{align*}
where $\mathrm{d}_\nu = -2\dim_\C \tF_\nu$ and
 $\mathrm{t}_{\nu,j} = \left\{
             \begin{array}{ll}
               2(\ell_{\nu_j}-1) & \hbox{ if } \nu_j = \nu_{j+1},  \\
               2h_{\nu_{j}, \nu_{j+1}} & \hbox{ if } \nu_j \ne \nu_{j+1}.
             \end{array}
           \right.
$
\end{Lem}
\begin{proof} It follows directly from $\eqref{Eq: def of e and x}$ that $e(\nu) \in H^{\LG_\alpha^\Omega}_{\mathrm{d}_\nu}(\vZ_{\nu,\nu})$ and  $\varkappa_i(\nu) \in H^{\LG_\alpha^\Omega}_{\mathrm{d}_\nu+2}(\vZ_{\nu,\nu})$.
We focus on the assertion $\tau_j(\nu) \in
H^{\LG_\alpha^\Omega}_{\mathrm{d}_{s_j\nu} +
\mathrm{t}_{\nu,j}}(\vZ_{s_j\nu,\nu})$.
Set
$$\overset{\circ}{\vZ} \seteq \mathrm{q}_{s_j\nu, \nu}^{-1}( \vO^{s_j}_{s_j\nu, \nu} ).$$
Since $\overset{\circ}{\vZ}$ is a dense open subset of
$\vZ^{ s_j}_{s_j\nu, \nu}$, it suffices to show that
$$ \dim_\C \overset{\circ}{\vZ}
=\begin{cases}
\dim_\C \tF_\nu - \ell_{\nu_j} + 1 & \text{if $\nu_j = \nu_{j+1}$,} \\
\dim_\C \tF_{s_j\nu} - h_{\nu_{j}, \nu_{j+1}} & \text{if $\nu_j \ne \nu_{j+1}$.}
\end{cases}
 $$
By $\eqref{eq:cw}$, we have
\begin{align} \label{Eq: dim of tF}
\dim_\C \tF_{s_j\nu} = \dim_\C \F_{s_j\nu} +
\sum_{k > k'} \#\Omega_{(s_j\nu)_k,(s_j\nu)_{k'}}.
\end{align}
We now consider the following maps
\begin{align} \label{Eq: bundle map Y}
\overset{\circ}{\vZ} \buildrel \mathrm{q}_{s_j\nu,\nu}\over \longrightarrow \vO^{s_j}_{s_j\nu, \nu} \buildrel \mathrm{pr}_{1}\over \longrightarrow \F_{s_j\nu} ,
\end{align}
where $\mathrm{pr}_1$ is the first projection.
The fiber of the vector bundle
$\overset{\circ}{\vZ} \longrightarrow \vO^{s_j}_{w s_j, w}$
is isomorphic to $\ve_{ws_j}\cap\ve_w$, and  $\eqref{eq:cw}$
implies that its dimension
is equal to
\begin{align} \label{Eq: dim of fiber of Y}
\sum_{k > k',\ (k,k') \ne (j+1,j)}\#\Omega_{(s_j\nu)_{k},(s_j\nu)_{k'}}.
\end{align}
Hence we have
\eqn
&&\dim_\C\overset{\circ}{\vZ}=\dim_\C\vO^{s_j}_{s_j\nu, \nu}+
\sum_{k> k',\ (k,k') \ne (j+1,j)}\# \Omega_{(s_j\nu)_{k},(s_j\nu)_{k'}}
\eneqn
Together with \eqref{Eq: dim of tF} we obtain
\eqn
\dim_\C\overset{\circ}{\vZ}-\dim_\C \tF_{s_j\nu}=\dim_\C\vO^{s_j}_{s_j\nu, \nu}
-\dim_\C \F_{s_j\nu}-\#\Omega_{(s_j\nu)_{j+1},(s_j\nu)_{j}}.
\eneqn
Hence our
claim follows from the fact
$$\dim_\C\vO^{s_j}_{s_j\nu, \nu}=\dim_\C \F_{s_j\nu}+\delta_{s_j\nu,\nu}$$
in Lemma~\ref{Lem: properties of O} \eqref{Oitem3}.
\end{proof}

\subsection{Localization}
Let us recall the localization theorem in equivariant cohomologies.
We refer the reader to
\cite{Brion98,Fulton07}.

Let $\gP = H_{\LT^\Omega}^*({\rm pt})$ and recall the ring $\gS =
H^*_{\LG_\alpha^\Omega}( \rm pt)$ given in $\eqref{Eq: def of S}$.
We have
$$\gP \simeq \gH \otimes H_{\LT }^*({\rm pt}) \simeq \gH[\chi_1, \ldots, \chi_m] $$
with $\deg \chi_i=2$ for $i=1,\ldots,m$.
The action of $\weyl$ on $\LT$ induces the action of $\weyl$ on $\gP$; i.e.,
$$ w \bl f(\chi_1, \ldots, \chi_m ) \br= f(\chi_{w(1)}, \ldots, \chi_{w(m)} ) $$
for $w\in \weyl$, $ f(\chi_1, \ldots, \chi_m ) \in \gH[\chi_1,
\ldots, \chi_m] $. Then the group morphism $\LT
\rightarrow \LG_\alpha$ induces the homomorphism
$\gS =H^*_{\LG_\alpha^\Omega}(\mathrm{pt}) \rightarrow
H^*_{\LT^\Omega}(\mathrm{pt}) $ and $\gS$ can be identified with
$$ \gS \simeq \gH[\chi_1, \ldots, \chi_m]^{\weyl_\alpha} \hookrightarrow \gP,$$
where $\gH[\chi_1, \ldots, \chi_m]^{\weyl_\alpha}$ is the set of
$\weyl_\alpha$-invariant polynomials in $\gH[\chi_1, \ldots,
\chi_m]$.
Let $\gK$ be the
fraction field of $\gP$:
$$\gK = \C(\chi_1, \ldots, \chi_m, \hbar_a\ (a \in \Omega))$$
and consider $\gP$ as a subring of $\gK$.

Let $X$ be a  quasi-projective $\LT^\Omega$-variety.
Then the inclusion $\iota\cl X^{\LT^\Omega}\hookrightarrow X$
induces isomorphisms (localization theorem):
\begin{align}\label{eq:lovcalization_hom}
\gK \otimes_{\gP}  H_*^{\LT^\Omega}(X^{\LT^\Omega}) &\isoto[\iota_*]
\gK \otimes_{\gP} H_*^{\LT^\Omega}(X)\\
\intertext{and}
\gK \otimes_{\gP}  H^*_{\LT^\Omega}(X) &\isoto[\iota^*]
\gK \otimes_{\gP} H^*_{\LT^\Omega}(X^{\LT^\Omega}).
\label{eq:lovcalization_coh}
\end{align}

Let $\Lt^\Omega \seteq \Lt \oplus \Lh^\Omega $ be the Lie algebra of the
group $\LT^\Omega = \LT \times \LH^\Omega$. For a finite-dimensional
weight module $M$ over $\Lt^\Omega$, set 
$$ \mathfrak{d}(M) = \prod_{\mu} \mu ^{ \dim_\C M_\mu} \in \gP, $$
where $M = \bigoplus_\mu M_\mu $ is the weight space
decomposition of $M$.

Assume that
\eq\hs{-1ex}
\text{the fixed point set $X^{\LT^\Omega}$ is finite
and consists of smooth points of $X$.}\label{cond:finite}
\eneq
For a smooth $\LT^\Omega$-fixed point $p$,
the {\em equivariant Euler class} $\eu( X, p )$ is by definition
the image of $1\in\gP$ by the composition
$\gP\simeq H^*_{\LT^\Omega}(p)\longrightarrow H^*_{\LT^\Omega}(X)\lan2\dim_\C X\ran
\longrightarrow H^*_{\LT^\Omega}(p)\lan2\dim_\C X\ran\simeq\gP\lan2\dim_\C X\ran$
where the first arrow is the Gysin map.
Then we have (see e.g.\ \cite{Brion98})
\eq&&\eu( X, p )=\mathfrak{d}(T_{p} X).\eneq

The localization theorem (see e.g.\ \cite{Brion98}) says
\eq &&[X]=\sum_{p\in X^{\LT^\Omega}}\eu( X, p )^{-1}[p]
\quad\text{in $\gK \otimes_{\gP} H_*^{\LT^\Omega}(X)$.}
\label{eq:localize}\eneq
 Note that $\eu( X, p )$ never vanishes.
Note also that, under the condition \eqref{cond:finite},
the Gysin morphism induces an isomorphism:
\eq
&&\gK \otimes_{\gP}  H^*_{\LT^\Omega}(X^{\LT^\Omega}) \isoto
\gK \otimes_{\gP} H^*_{\LT^\Omega}(X)\lan2\dim_\C X\ran.
\label{eq:Gysin}\eneq

\bigskip
We now apply the localization theorem
to $\LT^\Omega$-varieties $\tF_\alpha$ and $\vZ_\alpha$.
By \cite[Proposition 1.2.1]{Brion04} and Lemma \ref{Lem: properties of O} (1),
the sets $\tF_\alpha^{\LT^\Omega}$ and $\vZ_\alpha^{\LT^\Omega}$ of $\LT^\Omega$-fixed points are given by
\begin{align*}
\tF_\alpha^{\LT^\Omega} &= \{(0,F_w ) \mid w \in \weyl \} \subset
\tF_\alpha, \quad \tF_\nu^{\LT^\Omega} = \tF_\alpha^{\LT^\Omega}
\cap \tF_\nu,   \\
\vZ_\alpha^{\LT^\Omega}  &= \{(0,F_{w,w'} ) \mid w, w' \in \weyl \}
\subset \vZ_\alpha,  \quad \vZ_{\nu, \nu'}^{\LT^\Omega} =
\vZ_\alpha^{\LT^\Omega} \cap \vZ_{\nu, \nu'}
\end{align*}
for $\nu, \nu' \in I^\alpha$.
Then we have $\gK$-module isomorphisms by
\eqref{eq:lovcalization_hom} and \eqref{eq:Gysin}
\begin{equation} \label{Eq: localization}
\begin{aligned}
\gK \otimes_{\gP} H^{*}_{\LT^\Omega}(\tF_\nu^{\LT^\Omega})
 &\isoto \gK \otimes_{\gP} H^*_{\LT^\Omega}(\tF_\nu)\lan2\dim_\C\tF_\nu\ran  , \\
 \gK \otimes_{\gP} H_{*}^{\LT^\Omega}(\vZ_{\nu, \nu'}^{\LT^\Omega})
&\isoto \gK \otimes_{\gP} H_*^{\LT^\Omega}(\vZ_{\nu, \nu'}),
\end{aligned}
\end{equation}
which yields
\begin{align} \label{Eq: localization isom}
\gK \otimes_{\gP} H^*_{\LT^\Omega}(\tF_\nu) \simeq \bigoplus_{w \in
\weyl} \gK \ \zeta_w, \quad \gK \otimes_{\gP}H_*^{\LT^\Omega}(\vZ_{\nu, \nu'}) \simeq \bigoplus_{w,w' \in \weyl}
\gK \ \zeta_{w,w'}.
\end{align}
Here, $\zeta_w$ (resp.\ $\zeta_{w,w'}$) is the element in
$H^{*}_{\LT^\Omega}(\tF_\nu)$ (resp.\
$H_{*}^{\LT^\Omega}(\vZ_{\nu,\nu'})$) which is the image of $(0,F_w
) \in H^{*}_{\LT^\Omega}(\tF_\nu^{\LT^\Omega})$ (resp.\ $(0,F_{w,w'}
) \in H_{*}^{\LT^\Omega}(\vZ_{\nu, \nu'}^{\LT^\Omega})$) under the
isomorphisms $\eqref{Eq: localization}$. We have the following
injective $\gP$-module homomorphisms \cite[(2.3)]{VV11}
\begin{equation} \label{Eq: def of Psi and Phi}
\begin{aligned}
\Psi_{\nu}& : H^*_{\LG_\alpha^\Omega}(\tF_{\nu} )\hookrightarrow \gK \otimes_{\gP}
H^*_{\LT^\Omega}(\tF_{\nu}), \\
\Phi_{\nu,\nu'}& : H_*^{\LG_\alpha^\Omega}(\vZ_{\nu,\nu'} ) \hookrightarrow \gK \otimes_{\gP} H_*^{\LT^\Omega}(\vZ_{\nu, \nu'})
\end{aligned}
\end{equation}
for $\nu, \nu' \in I^\alpha$. Let $\Psi_{\alpha} = \sum_{\nu \in
I^\alpha}\Psi_{\nu}$ and $\Phi_{\alpha} = \sum_{\nu, \nu' \in
I^\alpha}\Phi_{\nu,\nu'}$. Then we obtain the following commutative
diagram:
  \begin{equation} \label{Eq: localization diagram1}
  \begin{aligned}
\xymatrix{
  H_*^{\LG_\alpha^\Omega}(\vZ_\alpha) \times H^*_{\LG_\alpha^\Omega}(\tF_\alpha) \ar@{^(->}[d] \ar[r]^{\qquad \ \ \star} &
H^*_{\LG_\alpha^\Omega}(\tF_\alpha) \ar@{^(->}[d] \\
  \gK \otimes_{\gP} H_*^{\LT^\Omega}(\vZ_\alpha) \times  \gK \otimes_{\gP} H^*_{\LT^\Omega}(\tF_\alpha)  \ar[r]^{\qquad \qquad  \star} &   \gK \otimes_{\gP}
 H^*_{\LT^\Omega}(\tF_\alpha).
}
\end{aligned}
\end{equation}
Here, $\star$ is the convolution product.

Set
\begin{alignat*}{2}
\Lambda_w &= \eu( \tF_\alpha, (0,F_w) )&\quad&\text{for $w\in \weyl$ and,}\\
\Lambda^{s_j}_{w,w'}&=\eu(\vZ^{s_j}_{w,w'}, (0,F_{w,w'}) )&\quad&
\parbox[t]{50ex}{for $j=1,\ldots,m-1$
and\\$w,w'\in \weyl$ such that $w'=w, ws_j$ and $\nu_w=\nu_{w's_j}$.
}
\end{alignat*}

Then $\Lambda_w$ and $\Lambda^{s_j}_{w,w'}$
are elements of $\gP$ of
degree $2\dim_\C \tF_{\nu_w}$ and $2\dim \vZ^{s_j}_{\nu_w,\nu_{w'}}$,
respectively.
 Note that $\tF_\alpha$ and $\vZ^{s_j}_{w,w'}$ are
smooth varieties with finitely many $\LT^\Omega$-fixed points
(see Lemma~\ref{lem:Zsj}).

 By Lemma \ref{Lem: properties of O} and \eqref{eq:localize},
we obtain the following lemma (see also \cite[Lemma 2.17, Lemma 2.19]{VV11}).

\begin{Lem} \label{Lem: convolution} \
\bnum
\item
For $ \nu \in I^\alpha$ and $f \in \gH[x_1, \ldots,x_m]$, we have
$$ \Psi_\nu\bl f(\chi_1(\nu), \ldots, \chi_m(\nu))\br = \sum_{w \in \weyl(\nu)} f(\chi_{w(1)}, \ldots, \chi_{w(m)}) \Lambda_w^{-1} \zeta_w, $$
where $\weyl(\nu)\seteq\{ w \in \weyl \mid F_w\in\F_\nu\} =
 \{ w \in \weyl \mid \nu_w = \nu  \}$.
\item For $ \nu \in I^\alpha$ and $f \in \gH[x_1, \ldots,x_m]$,
$$\Phi_{\nu,\nu}\bl f(\varkappa_1(\nu),\ldots,\varkappa_m(\nu))\br
=\sum_{w \in \weyl(\nu)} f(\chi_{w(1)}, \ldots, \chi_{w(m)})
 \Lambda_w^{-1} \zeta_{w,w}, $$
\item
For $j=1,\ldots, m-1$ and $\nu \in I^\alpha$,
$$  \Phi_{\alpha} (\tau_j(\nu)) =
\begin{cases}
\sum\limits_{w\in \weyl(\nu)} \Lambda^{s_j}_{ws_j,w}{}^{-1} \zeta_{ws_j,w}
& \text{if $s_j\nu \ne \nu$,} \\
\sum\limits_{w\in \weyl(\nu)}\bl \Lambda^{s_j}_{w,w}{}^{-1} \zeta_{w,w}
+ \Lambda^{s_j}_{ws_j,w}{}^{-1} \zeta_{ws_j,w} \br
& \text{if $s_j\nu = \nu$.}
\end{cases}
$$
\item For $w,w',w'' \in \weyl$, we have
$$ \zeta_{w,w'} \star \zeta_{w''} = \delta_{w',w''}\Lambda_{w''} \zeta_{w}.$$
\end{enumerate}
\end{Lem}
\begin{proof}
(i)\ For any $w\in W(\nu)$, let $R_w$ be the map
$\gK \otimes_{\gP} H^{*}_{\LT^\Omega}(\tF_\nu)
\to \gK \otimes_{\gP} H^*_{\LT^\Omega}((0,F_w))\simeq\gK$
induced by the inclusion $(0,F_w)\hookrightarrow \tF_w$.
Then the map
$$\gK \otimes_{\gP} H^{*}_{\LT^\Omega}(\tF_\nu)\To[{\oplus R_w}]
\soplus_{w\in W(\nu)}\gK \otimes_{\gP} H^*_{\LT^\Omega}((0,F_w))
\simeq  H^*_{\LT^\Omega}(\tF_\nu^{\LT^\Omega})$$
is an isomorphism (see \eqref{eq:lovcalization_coh}).
Then (i) follows from
$$R_{w}\Bigl(\Psi_\nu\bl f(\chi_1(\nu), \ldots, \chi_m(\nu))\br\Bigr)
=f(\chi_{w(1)}, \ldots, \chi_{w(m)})$$ (see \eqref{eq:locchi}) and
$R_w(\zeta_{w'})=\delta_{w,w'}\Lambda_w$.

\smallskip
\noi
(ii) is similarly proved.

\smallskip\noi
(iii) immediately follows from \eqref{eq:localize} with $X=\vZ_{s_j\nu,\nu}^{s_j}$.

\smallskip\noi
(iv) It is obvious that $\zeta_{w,w'} \star \zeta_{w''} = 0$
as soon as $w'\not=w''$.
Hence we have
$$\zeta_{w}= \zeta_{w,w'} \star[\tF_\alpha]
=\zeta_{w,w'}\star\bl\sum_{w''\in W}\Lambda_{w''}^{-1}\zeta_{w''}\br
=\zeta_{w,w'}\star\bl\Lambda_{w'}^{-1}\zeta_{w'}\br.$$
Here the second equality follows from \eqref{eq:localize}.
\end{proof}

Recall that $\ve_w$ is the fiber of the vector bundle
$\tF_\alpha\to\F_\alpha$ at $F_w$
(see \eqref{Eq: def of ew}).
Hence we have
\begin{equation}\label{Eq: Lambda 1}
\Lambda_w = \eu( \tF_{\nu_w}, (0, F_{w} ))
 = \eu( \F_{\nu_w}, F_{w} ) \ \mathfrak{d}(\ve_{w}).
\end{equation}

Let $w,w'\in \weyl$ such that $F_{w', w}\in\overline{\vO}^{s_j}_{\alpha}$,
\i.e., $w'=w,\,ws_j$ and $\nu_w=\nu_{w's_j}$
(see Lemma~\ref{Lem: properties of O}).
Then $\vZ_{\alpha}^{s_j}\to\overline{\vO}^{s_j}_{\alpha}$
is a vector bundle and its fiber at $F_{w, w'}$ is $\ve_{w}\cap\ve_{ws_j}$
(see Lemma~\ref{lem:Zsj}).
Hence we have
\begin{equation} \label{Eq: Lambda 2}
\begin{aligned}
{\Lambda_{w,w'}^{s_j}}&= \eu(\vZ_{\alpha}^{s_j}, (0, F_{w, w'}) ) \\
&=\eu(\overline{\vO}^{s_j}_{\alpha},  F_{w, w'} )\
\mathfrak{d}(\ve_{w} \cap \ve_{ws_j}).
\end{aligned}
\end{equation}

In the following lemma, we compute the quotients
$\Lambda_{w,w}^{s_j}{}^{-1}\ \Lambda_w$ and $\Lambda_{w,ws_j}^{s_j}{}^{-1}\
\Lambda_w$ of the equivariant Euler classes
in order to describe explicitly the actions of $e(\nu)$, $\varkappa_k(\nu)$
and $\tau_t(\nu)$ on $\gPol(\alpha)$. The polynomials
$\mathcal{P}_i$ and $\mathcal{Q}_{i,j}$ given in $\eqref{Eq: P and
Q}$ arise naturally in the course of computation.

\begin{Lem} \label{Lem: computation for Lambdas}
Let $w\in \weyl$ and $j=1,\ldots,m-1$. Set $\nu =
\nu_w$ and write $\nu = (\nu_1, \nu_2, \ldots, \nu_m)$.
\bnum
\item If $ \nu \ne s_j\nu$, then
$$
\Lambda_{w,ws_j}^{s_j}{}^{-1}\Lambda_w = \prod_{a \in \Omega_{\nu_{j+1}, \nu_{j}}}(\chi_{w(j)} - \chi_{w(j+1)} + \hbar_a) ,
$$
\item If $ \nu = s_j\nu$, then
\begin{align*}
\Lambda_{w,w'}^{s_j}{}^{-1}\Lambda_w = (-1)^{\delta_{w,w'} }  \frac{\prod_{a \in \Omega_{\nu_j, \nu_j}}(\chi_{w(j)} - \chi_{w(j+1)} + \hbar_a)}{\chi_{w(j)} - \chi_{w(j+1)}},
\end{align*}
where $w' = w$ or $w' =  ws_j$.
\end{enumerate}
\end{Lem}
\begin{proof}
By \eqref{eq:cw}, we have
\begin{align} \label{Eq: frac of ew}
\frac{\ve_w}{\ve_w \cap \ve_{ws_j}} \simeq \bigoplus_{a \in \Omega_{\nu_{j+1}, \nu_j}} \mathrm{Hom}(\mathbb{F}_{w(j+1)}, \mathbb{F}_{w(j)} )
\end{align}
as a $\Lt^\Omega$-module.
Hence we have
\eqn
&&\mathfrak{d}(\frac{\ve_w}{\ve_w \cap \ve_{ws_j}})
=\prod_{a\in \Omega_{\nu_{j+1}, \nu_j}}(\chi_{w(j)}-\chi_{w(j+1)}+\hbar_a).
\eneqn

By  \eqref{Eq: Lambda 1} and \eqref{Eq: Lambda 2}, we have
\eqn
\Lambda_{w,w'}^{s_j}{}^{-1} \Lambda_w &=&
\frac{\eu(\F_w,F_w)}{\eu(\oO^{s_j}_{w , w'},F_{w,w'})}
\frac{ \mathfrak{d}(\ve_{w}) }{ \mathfrak{d}(\ve_{w} \cap \ve_{ws_j})}\\
&=&\frac{\eu(\F_w,F_w)}{\eu(\oO^{s_j}_{w , w'},F_{w,w'})}
\prod_{a\in \Omega_{\nu_{j+1}, \nu_j}}(\chi_{w(j)}-\chi_{w(j+1)}+\hbar_a).
\eneqn

\smallskip
\noi
(i) Assume that $\nu\not=s_j\nu$. Then $w'=ws_j$ and
$\oO^{s_j}_{w , ws_j}\isoto \F_{\nu_w}$ by Lemma~\ref{Lem: properties of O}
~\eqref{Oitem5}. Hence we obtain (i).

\smallskip
\noi
(ii) Assume that $\nu_w = s_j\nu_{w}$, and $w'=w,ws_j$.
Recall that $B_{\alpha,w}=\{g\in G_\alpha\mid gF_w=F_w\}$ and
$\bal$ is its Lie algebra.
The morphism $\oO^{s_j}_{w , w'}\to \F_w$ is a $\mathbb{P}^1$-bundle and
the tangent space of the fiber at $F_{w,w'}$
is isomorphic to
$$\dfrac{\bal+\bal[ws_j]}{\bal[w']}
\simeq
\begin{cases}
\mathrm{Hom}(\mathbb{F}_{w(j)}, \mathbb{F}_{w(j+1)} )&\text{if $w'=w$,}\\
\mathrm{Hom}(\mathbb{F}_{w(j+1)}, \mathbb{F}_{w(j)} )&\text{if $w'=ws_j$}
\end{cases}
$$
by \eqref{eq:bw}.
Hence we obtain
$$\frac{\eu(\oO^{s_j}_{w , w'},F_{w,w'})}{\eu(\F_w,F_w)}
=\begin{cases}
\chi_{w(j+1)}-\chi_{w(j)}&\text{if $w'=w$}\\
\chi_{w(j)}-\chi_{w(j+1)} &\text{if $w'=ws_j$,}
\end{cases}$$
which implies (ii).
\end{proof}

We now describe explicitly the $\gR(\alpha)$-module structure of
$\gPol(\alpha)$.
Recall that for $f\in\gH[x_1,\ldots,x_m]$ and $\nu\in I^\alpha$, we denote by
$f(\nu)$ the element
$f\bl\chi_1(\nu), \ldots, \chi_m(\nu)\br\in \gPol(\alpha)$.
Recall also that $(w f)(x_{1},\ldots, x_{m}) =
f(x_{w(1)}, \ldots, x_{w(m)})$ for $w\in
\weyl$.
The actions of $e(\nu)$, $\varkappa_k(\nu)$, $\tau_t(\nu) \in
\gR(\alpha)$ ($k=1,\ldots,m,\ t=1,\ldots,m-1,\ \nu \in I^\alpha$) on
$\gPol(\alpha)$ are given explicitly in the following
proposition.

\begin{Prop} \label{Prop: actions of Pol} \
\bnum
\item $\gPol(\alpha)$ is a faithful $\gR(\alpha)$-module.
\item Let $f \in \gH[\chi_1, \ldots, \chi_m]$, and $\nu, \nu' \in I^\alpha$.
\begin{enumerate}[{\rm(a)}]
\item For $k = 1,\ldots, m$, we have
\begin{align*}
\hs{8ex}e(\nu) \star f(\nu') = \left\{
              \begin{array}{ll}
                f(\nu) & \hbox{ if } \nu = \nu',  \\
                0 & \hbox{ if } \nu \ne \nu',
              \end{array}
            \right.
\qquad
\varkappa_k(\nu) \star f(\nu') = \left\{
              \begin{array}{ll}
               \chi_k(\nu) f(\nu) & \hbox{ if } \nu = \nu',  \\
                0 & \hbox{ if } \nu \ne \nu'.
              \end{array}
            \right.
\end{align*}
\item For $j = 1,\ldots, m-1$ and $\nu=(\nu_1, \ldots, \nu_m)$ , we have

\begin{align*}
\tilde\tau_j\star f(\nu) = \left\{
                   \begin{array}{ll}
                                          \left( \prod_{a\in \Omega_{\nu_j, \nu_{j+1}}}(\chi_{j}(s_j\nu) - \chi_{j+1}(s_j\nu)+ \hbar_a) \right) (s_jf)(s_j\nu) &
\text{if $s_j\nu \ne \nu$,} \\
\left( \prod_{a\in \Omega_{\nu_j, \nu_{j}}} (\chi_{j}(\nu) - \chi_{j+1}(\nu)+ \hbar_a )
\right) (\partial_j f)(\nu) & \text{if $s_j\nu = \nu$.}
                   \end{array}
                 \right.
\end{align*}
\end{enumerate}
\end{enumerate}
\end{Prop}
\begin{proof}
(i) Our assertion follows from Lemma \ref{Lem: convolution} (iv).

\smallskip
\noi
(ii) Since the assertion (a) is straightforward,
we shall prove the assertion (b).
Since the diagram $\eqref{Eq:
localization diagram1}$ is commutative, it suffices to show that
$$ \Phi_\alpha(\tilde\tau_j) \star \Psi_\alpha(f(\nu)) = \Psi_\alpha(g), $$
where $g = \left\{
             \begin{array}{ll}
               \left( \prod_{a\in \Omega_{\nu_j, \nu_{j+1}}}(\chi_{j}(s_j\nu) - \chi_{j+1}(s_j\nu)+ \hbar_a) \right) (s_jf)(s_j\nu) & \hbox{ if } s_j\nu \ne \nu,\\
               \left( \prod_{a\in \Omega_{\nu_j, \nu_{j}}} (\chi_{j}(\nu) - \chi_{j+1}(\nu)+ \hbar_a ) \right) (\partial_j f)(\nu) & \hbox{ if } s_j\nu = \nu.
             \end{array}
           \right.
$

\smallskip
By Lemma~\ref{Lem: convolution}, we have
\eq&&\ba{rcl}
\Phi_\alpha(\tilde\tau_j)
\star \Psi_\alpha(f(\nu))&=&
\sum\limits_{w\in \weyl}\bl
\delta_{\nu_w,\nu_{ws_j}}\Lambda^{s_j}_{w,w}{}^{-1} \zeta_{w,w}
+ \Lambda^{s_j}_{ws_j,w}{}^{-1} \zeta_{ws_j,w} \br
\star \sum\limits_{w \in \weyl(\nu)} (wf) \Lambda_w^{-1} \zeta_w \\
&=& \sum\limits_{w \in \weyl(\nu)}
\Bigl( \delta_{\nu_w,\nu_{ws_j}}\Lambda_{w,w}^{s_j}{}^{-1}(wf) \zeta_w
+ \Lambda_{ws_j,w}^{s_j}{}^{-1} (wf) \zeta_{ws_j} \Bigr).
\ea\label{eq:phipsi}
\eneq

Suppose that $ \nu \ne s_j\nu$.
Then, by Lemma \ref{Lem: computation for Lambdas} and \eqref{eq:phipsi},
we have
\begin{align*}
\Phi_\alpha(\tilde\tau_j) \star \Psi_\alpha(f(\nu)) &=
\sum_{w \in \weyl(\nu)} \Lambda_{ws_j,w}^{s_j}{}^{-1} (wf) \zeta_{ws_j}\\
&= \sum_{w \in \weyl(s_j\nu)}
\Lambda_{w,ws_j}^{s_j}{}^{-1}\Lambda_w(ws_jf) \bl\Lambda_w^{-1}\zeta_{w}\br\\
&= \sum_{w \in \weyl({s_j\nu})} \Bigl(\prod_{a \in \Omega_{\nu_{j}, \nu_{j+1}}}(\chi_{w(j)} - \chi_{w(j+1)} + \hbar_a)\Bigr) (ws_jf)  \Lambda_{w}^{-1} \zeta_w \\
&= \Psi_\alpha \biggl( \Bigl( \prod_{a\in \Omega_{\nu_j, \nu_{j+1}}}(\chi_{j}(s_j\nu) - \chi_{j+1}(s_j\nu)+ \hbar_a) \Bigr) (s_jf)(s_j\nu) \biggr).
\end{align*}
Here the last equality follows from Lemma~\ref{Lem: convolution} (i).

\smallskip
We now assume that $ \nu = s_j\nu$.
By Lemma \ref{Lem: computation for Lambdas} and \eqref{eq:phipsi}, we obtain

\begin{align*}
\Phi_\alpha(\tilde\tau_j) \star \Psi_\alpha(f(\nu)) &=
\sum_{w \in \weyl(\nu)}
\Bigl(\Lambda_{w,w}^{s_j}{}^{-1}(wf) \zeta_w
+ \Lambda_{ws_j,w}^{s_j}{}^{-1} (wf) \zeta_{ws_j} \Bigr)\\
 &= \sum_{w \in \weyl({\nu})} \Bigr( \Lambda_{w,w}^{s_j}{}^{-1}\Lambda_w (wf)  +
\Lambda_{w,ws_j}^{s_j}{}^{-1} \Lambda_w (ws_jf) \Bigr) \Lambda_w^{-1} \zeta_{w} \\
&= \sum_{w \in \weyl({\nu})}
\Bigl(\prod_{a\in \Omega_{\nu_j, \nu_{j}}} (\chi_{j}(\nu) - \chi_{j+1}(\nu)+ \hbar_a ) \Bigr)
\frac{ ws_jf - wf}{\chi_{w(j)}- \chi_{w(j+1)}}
\Lambda_w^{-1} \zeta_{w}  \\
&= \Psi_\alpha \biggl(  \Bigl(
\prod_{a\in \Omega_{\nu_j, \nu_{j}}} (\chi_{j}(\nu) - \chi_{j+1}(\nu)+ \hbar_a ) \Bigr)
\frac{ (s_jf)(\nu) - f(\nu)}{\chi_{j}(\nu)- \chi_{j+1}(\nu)} \biggr),
\end{align*}
which completes the proof.
\end{proof}

Let $R(\alpha)$ be the Khovanov-Lauda-Rouquier algebra over the
graded commutative ring $\gH$ associated with the data
$(\mathsf{A},\mathsf{P},\Pi,\Pi^\vee)$ and the polynomials
$\mathcal{P}_i(u,v), \mathcal{Q}_{i,j}(u,v) \in
\gH[u,v]$ defined in $\eqref{Eq: P and Q}$.  We take the
$\Z$-grading defined by $\eqref{Eq: grading of R}$.
Then we have a faithful graded polynomial representation $\aPol(\alpha)$ of $R(\alpha)$ given in Proposition \ref{Prop: faithful}.

Now we can state and prove the main result of this paper.

\begin{Thm} \label{Thm: main thm}
  There exists a unique $\gH$-algebra isomorphism
$\Theta: R(\alpha) \longrightarrow \gR(\alpha)$ such that
\eq&&
\ba{l}
\Theta(\mathsf{e}(\nu)) = e(\nu), \quad
\Theta(\mathsf{x}_k(\nu)) = \varkappa_k (\nu), \quad
\Theta(\mathsf{r}_t(\nu)) = \tau_t (\nu) \\[1ex]
\hs{19ex}\text{for $\nu \in I^\alpha$, $k=1,\ldots,m$ and $t=1,\ldots,m-1$.}
\ea
\label{eq:Theta}
\eneq
\end{Thm}
\begin{proof}
We can easily identify $\aPol(\alpha)$ with $\gPol(\alpha)$. It
follows from $\eqref{Eq: actions in aPol}$, Proposition \ref{Prop:
faithful} and Proposition \ref{Prop: actions of Pol} that
there exists a unique injective $\gH$-algebra homomorphism $\Theta$ satisfying
\eqref{eq:Theta}.
 By Lemma~\ref{Lem: inj of R},
$\gR(\alpha)$ is generated by $e(\nu)$, $\varkappa_k (\nu)$ and
$\tau_t (\nu)$, and hence $\Theta$ is surjective.
\end{proof}

\vskip 2em

\section{Indecomposable projective modules and lower global bases}
\addtocounter{subsection}{1}
In this section, we give a 1-1 correspondence between {\it
Kashiwara's lower global basis} (or {\it Lusztig's canonical basis})
of $U_\A^-(\g)$ (resp.\ $V_\A(\lambda)$) and the set of isomorphism
classes of indecomposable projective graded $R$-modules (resp.\
indecomposable projective graded $R^\lambda$-modules) when any
of the diagonal entries of the symmetric Borcherds-Cartan matrix $A=(a_{ij})_{i,j\in
I}$ does not vanish. Let us keep all the notations appeared in the
previous sections.  We first suppose that
the symmetric Borcherds-Cartan
matrix $A$ is arbitrary.

For a given quasi-projective variety $X$ over $\C$ with an action of
a complex linear algebraic group $G$ and $A,B \in \D^b_{\LG}(X) $,
let $\Hom_{\LG}^{k} (A,B) = \Hom_{\D^b_{\LG}(X)}(A, B [k])$, where
$[ \cdot ]$ is the translation functor.
Recall the map $\pi_\nu$ given in $\eqref{Eq: maps pi and p}$.
For $\nu \in I^{\alpha}$, we write
\begin{align*}
\mathcal{L}_\nu = R{\pi_\nu}_! (\C_{\tF_\nu} [2\dim_\C \tF_\nu ])
\in  \D_{\LG_\alpha^\Omega}^b(\E_\alpha).
\end{align*}
Note that $\mathcal{L}_\nu$ is semisimple \cite[Proposition
4.1]{KS06}. Let $\mathcal{L}_\alpha = \bigoplus_{\nu\in
I^\alpha}\mathcal{L}_\nu$. Then, by the same argument as in the
proof of \cite[Lemma 8.6.1]{CG97}, we obtain
\begin{align} \label{Eq: Ext isom}
\gR({\alpha}) \simeq  \Hom_{\LG_\alpha^\Omega}^*( \mathcal{L}_\alpha, \mathcal{L}_{\alpha} )
\quad\text{and}\quad
 \gR_{\nu,\nu'} \simeq  \Hom_{\LG_\alpha^\Omega}^*( \mathcal{L}_\nu, \mathcal{L}_{\nu'} )
\quad\text{for $\nu, \nu' \in I^\alpha$.}
\end{align}
Then $\gR(\alpha)$ is isomorphic to
the opposite algebra $\Hom_{\LG_\alpha^\Omega}^*(\mathcal{L}_\alpha,
\mathcal{L}_\alpha)^{\rm op}$ of
$\Hom_{\LG_\alpha^\Omega}^*(\mathcal{L}_\alpha, \mathcal{L}_\alpha)$
\cite[Section 8.6]{CG97}.

For $\alpha \in \mathsf{Q}^+$, let $\mathcal{P}_\alpha$ be the set
of isomorphism classes of simple $\LG_\alpha^\Omega$-equivariant
perverse sheaves $\mathcal{L}$ on $\E_\alpha$ such that
$\mathcal{L}[ k ]$ appears as a direct summand of
$\mathcal{L}_\alpha$ for some $ k\in \Z$. Let $\mathcal{Q}_\alpha$
be the full subcategory of $\D_{\LG_\alpha^\Omega}^b(\E_\alpha)$
consisting of $\mathcal{L}$ having the form
$$ \mathcal{L} \simeq \mathcal{L}_1 [ k_1 ] \oplus \cdots  \oplus \mathcal{L}_r [ k_r ] $$
for some $\mathcal{L}_i\in \mathcal{P}_\alpha$ and $k_i\in \Z$.
Then $\mathcal{L}_\alpha$ belongs to $\mathcal{Q}_\alpha$
by the decomposition theorem \cite{BBD82}.

We now identify $R(\alpha)$ with $\gR(\alpha)$ via Theorem \ref{Thm:
main thm}. Then,  for $\mathcal{L} \in
\mathcal{Q}_\alpha$, the vector space $ \Hom_{\LG_\alpha^\Omega}^*(
\mathcal{L}_\alpha, \mathcal{L} )$ has the left $R(\alpha)$-module
structure. 
Moreover, $\Hom_{\LG_\alpha^\Omega}^*( \mathcal{L}_\alpha, \mathcal{L})$ is a
projective $R(\alpha)$-module by the construction of
$\mathcal{Q}_\alpha$.  Then we obtain the following proposition
by general results on idempotent complete categories.
\begin{Prop} \label{Prop: isom Upsilon}
Let
$$
\Upsilon_\alpha : \mathcal{Q}_\alpha \longrightarrow R(\alpha)\text{-}\proj
\quad \text{and}\quad \Xi_\alpha: R(\alpha)\text{-}\proj \longrightarrow \mathcal{Q}_\alpha
$$
 be the functors given by
\begin{align*}
\Upsilon_\alpha(\mathcal{L}) = \Hom_{\LG_\alpha^\Omega}^*( \mathcal{L}_\alpha, \mathcal{L} )\quad \text{and}\quad \Xi_\alpha(P) = \mathcal{L}_\alpha \otimes_{R(\alpha)}P
\end{align*}
for $\mathcal{L} \in \mathcal{Q}_\alpha$ and
$P \in R(\alpha)\text{-}\proj$, respectively.

Then $\Upsilon_\alpha$ is an equivalence of categories and $\Xi_\alpha$ is
its quasi-inverse.
\end{Prop}

From now on, we assume that $a_{ii} \ne 0$ for any $i\in I$. Let
$K(\mathcal{Q}_\alpha)$ denote the $\A$-module generated by
$[\mathcal{L}]$ for $\mathcal{L} \in \mathcal{Q}_\alpha$ subject to
the relations $[\mathcal{L}_1 \oplus \mathcal{L}_2] =
[\mathcal{L}_1]+[\mathcal{L}_2]$ and $[\mathcal{L} [ k ] ] = q^{-k}
[\mathcal{L}]$ for $k\in \Z$.

Let
$$ \mathcal{U}^-_{\A} = \bigoplus_{\alpha \in \mathsf{Q}^+} K(\mathcal{Q}_\alpha). $$
Then, together with the induction and restriction functors given in
\cite{KS06,Lus93}, $\mathcal{U}^-_{\A}$  becomes an $\A$-bialgebra
and there exists an isomorphism \cite[Section 5]{KS06}
\begin{align} \label{Eq: isom Psi}
 \Psi\cl U_{\mathbb{A}}^-(\g)\isoto \mathcal{U}^-_{\A},
\end{align}
where $U_{\mathbb{A}}^-(\g)$ is the $\A$-form of the negative half
of the quantum group $U_q(\g)$. Set
\begin{align}
\mathbf{B} = \bigsqcup_{\alpha \in \mathsf{Q}^+} \Psi^{-1}(\mathcal{P}_\alpha).
\end{align}
The set $\mathbf{B}$ coincides with
the lower global basis (or the canonical basis) of $U_{\A}^-(\g)$.

Let $K_0(R) = \bigoplus_{\alpha \in
\mathsf{Q}^+}K_0(R(\alpha)\text{-}\proj)$ be the Grothendieck group
given in $\eqref{Eq: Grothedieck gp}$ and define
$$\Upsilon = \bigsqcup_{\alpha \in \mathsf{Q}^+} \Upsilon_\alpha :
\mathcal{U}^-_{\A} \overset {\sim} \longrightarrow K_0(R),$$ where
$\Upsilon_\alpha$ is the isomorphism induced by Proposition
\ref{Prop: isom Upsilon}. Using the same argument as in
\cite[Section 4.6]{VV11}, it follows from Theorem \ref{Thm:
categorification U} and $\eqref{Eq: isom Psi}$ that
the isomorphism $\Phi\cl U_\A^-(\g) \rightarrow K_0(R)$ given
in Theorem \ref{Thm: categorification U} satisfies
\begin{align} \label{Eq: upsilon}
\Phi=\Upsilon \circ \Psi.
\end{align}
 Consequently, the first
application of our main result follows.

\begin{Thm} \label{Thm: lgb and PIM for U} The isomorphism $\Phi$ gives a 1-1
correspondence between $\mathbf{B}$ and the set of isomorphism
classes of indecomposable projective $R$-modules.
\end{Thm}

Let $\lambda$ be a dominant integral weight in $\mathsf{P^+}$ and $K_0(R^\lambda) = \bigoplus_{\alpha \in \mathsf{Q}^+}K_0(R^\lambda(\alpha)\text{-}\proj)$ be the Grothendieck group given in $\eqref{Eq: Grothedieck gp}$.
Let $v_\lambda$ be the highest weight vector of $V_{\A}(\lambda)$ and define $\mathbf{p}_\lambda: U_{\A}^-(\g) \longrightarrow V_{\A}(\lambda)$ by
$$   \mathbf{p}_\lambda(x) = xv_\lambda $$
for $x\in U_{\A}^-(\g)$. Similarly, we define $\mathbf{q}_\lambda: K_0(R) \longrightarrow K_0(R^\lambda)$ by
$$\mathbf{q}_\lambda(P) = R^\lambda(\alpha)\otimes_{R(\alpha)} P$$
for $P\in R(\alpha)\text{-}\proj$.
Then we have the following commutative diagram:
$$
\xymatrix{
 U_{\A}^-(\g) \ar[d]^{\mathbf{p}_\lambda} \ar[r]^{\Phi  } &  K_0(R) \ar[d]^{\mathbf{q}_\lambda} \\
 V_\A(\lambda) \ar[r]^{\Phi^\lambda} &  K_0(R^\lambda),
}
$$
where $\Phi^\lambda: V_\A(\lambda) \rightarrow K_0(R^\lambda)$ is
the isomorphism given in Theorem \ref{Thm: categorification V}.
Then $\mathbf{B}^\lambda \seteq\mathbf{p}_\lambda(\mathbf{B}) \setminus
\{0\}$ coincides with the {\it lower global basis} (or the {\it
canonical basis}) of $V_\A(\lambda)$. For $P\in
R(\alpha)\text{-}\proj$, since $R^\lambda(\alpha)\otimes_{R(\alpha)}
P$ can be viewed as an $R(\alpha)$-module, we have a surjective $
R(\alpha)$-module homomorphism
$$ P \twoheadrightarrow R^\lambda(\alpha)\otimes_{R(\alpha)} P, $$
which implies that $\mathbf{q}_\lambda$ takes an indecomposable
projective $R(\alpha)$-module to an indecomposable projective
$R^\lambda(\alpha)$-module or 0. Therefore we obtain the second
application of our main result.

\begin{Cor} \label{Cor: lgb and PIM for V}
The isomorphism $\Phi^\lambda$ gives a 1-1 correspondence between
$\mathbf{B}^\lambda$ and the set of isomorphism classes of
indecomposable projective $R^\lambda$-modules.
\end{Cor}

\vskip 2em


\bibliographystyle{amsplain}


\end{document}